\documentclass[12pt]{article}

\usepackage{indentfirst}
\usepackage{amsfonts}
\usepackage{amsmath}
\usepackage{amssymb}
\usepackage{amsbsy, amsthm}
\usepackage[margin=1in]{geometry}
\usepackage{imakeidx,xcolor}
\usepackage[noabbrev,capitalise]{cleveref}
\usepackage{thmtools}
\usepackage{thm-restate}

\newtheorem{dfn}{Definition}[section]
\newtheorem{theorem}[dfn]{Theorem}
\newtheorem{lemma}[dfn]{Lemma}

\newtheorem{corollary}[dfn]{Corollary}
\newtheorem{question}[dfn]{Question}

\newtheorem{obs}[dfn]{Observation}

\newcommand{\mc}[1]{\mathcal{#1}} 
\newcommand{\bb}[1]{\mathbb{#1}}
\newcommand{\brm}[1]{\operatorname{#1}}
\newcommand{\bd}[1]{\partial #1}

\newcommand{\fs}[2]{\left(\frac{#1}{#2}\right)}
\newcommand{\s}[1]{\left(#1\right)}

\newcommand{\diam}{\operatorname{diam}}
\newcommand{\tw}{\operatorname{tw}}

\newcommand{\asdim}{\operatorname{asdim}}
\newcommand{\iasdim}{\operatorname{int-asdim}}

\def\F {{\mathcal F}}
\def\Se {{\mathcal S}}
\def\A {{\mathcal A}}
\def\B {{\mathcal B}}

\def\U {{\mathcal U}}

\def\M {{\mathcal M}}

\def\dist {{\rm dist}}

\begin{document}

\title{Nerve-type and invariance theorems for asymptotic dimension}
\author{Chun-Hung Liu\thanks{chliu@tamu.edu. Department of Mathematics, Texas A\&M University, USA. Partially supported by NSF under CAREER award DMS-2144042.}
\and 
Sergey Norin\thanks{sergey.norin@mcgill.ca. Department of Mathematics and Statistics, McGill University, Canada. Supported by an NSERC Discovery grant.}}
\date{}
	\maketitle

\begin{abstract}
Asymptotic dimension of metric spaces is a large-scale analog of covering dimension of topological spaces.
An intersection graph of a family of sets is the graph whose vertices are the members of the family and whose edges correspond to pairs of members with non-empty intersection.

Our first main result connects the asymptotic dimension of the intersection graph of a family ${\mathcal F}$ and the Assouad-Nagata dimension of the ambient metric space containing members of ${\mathcal F}$ under some mild and necessary assumptions.
We prove that if ${\mathcal F}$ is a family of subsets of a metric space of Assouad-Nagata dimension $n$ such that every ball of radius $r$ intersects at most $f(r/s)$ pairwise disjoint members of ${\mathcal F}$ of diameter at least $s$ for some function $f$, then the asymptotic dimension of the intersection graph of ${\mathcal F}$ is at most $n+1$.
This result is optimal both quantitatively and qualitatively in several senses. 
As a corollary of this result, the asymptotic dimension of the intersection graph of any family of compact convex sets of bounded aspect ratio in ${\mathbb R}^n$, such as a family of balls in ${\mathbb R}^n$, is at most $n+1$. 

Our second main result states that the asymptotic dimension of the intersection graph of a family ${\mathcal F}$ of connected closed sets of a connected topological space with connected boundary equals the asymptotic dimension of the intersection graph of the family of the boundary of the sets in ${\mathcal F}$, under a mild condition.
In particular, the asymptotic dimension of the intersection graphs of families of spheres in ${\mathbb R}^n$ equals $n$ or $n+1$ when $n \geq 2$.
\end{abstract}

\section{Introduction}	

One way to measure the complexity of a graph is to study the complexity of objects that can represent the graph in terms of their pairwise intersection. 
Formally, for a family\footnote{Every family is a multiset in this paper.} $\F$ of sets, the \emph{intersection graph of $\F$} is the (finite or infinite) graph\footnote{All graphs are finite and simple in this paper unless otherwise specified. 
}, denoted by $I(\F)$, with vertex set $\F$ such that any two distinct vertices $S_1,S_2 \in \F$ are adjacent in $I(\F)$ if and only if $S_1 \cap S_2 \neq \emptyset$. 
For example, the \emph{boxicity} of a graph $G$ is the minimum $n$ such that $G$ is an intersection graph of a family of closed axis-parallel boxes (i.e. Cartesian products of closed intervals in $\mathbb{R}$) in ${\mathbb R}^n$. 
Boxicity has a tight connection to dimension of posets \cite{ABC11}, and the graphs of boxicity one are exactly the interval graphs.
Another example is Koebe's theorem, stating that the planar graphs are precisely the intersection graphs of the closed disks in ${\mathbb R}^2$ with pairwise disjoint interior.
More generally, intersection graphs of common geometric objects, such as balls or cubes in Euclidean spaces, have been extensively studied.
On the other hand, every graph is an intersection graph of  some family of sets, so we have to impose conditions on those sets in order to prove interesting results.

In this paper, we study asymptotic dimension of intersection graphs by proving two main theorems. 
The first can be seen as an analogue of nerve theorems for asymptotic dimension.  It relates the Assouad-Nagata dimension of the ambient metric space to the asymptotic dimension of the intersection graph of ``sufficiently nice'' families of  subsets of this space.
The second one relates the asymptotic dimension of an intersection graph of a family $\F$ of subsets of a topological space and the asymptotic dimension of the intersection graph of the families of the boundary of the members of $\F$.

Asymptotic dimension, introduced by Gromov \cite{G93}, addresses large-scale behaviors of a metric space and is an analog of covering dimension of topological spaces that addresses small-scale behaviors.
Assouad-Nagata dimension is a refinement of asymptotic dimension and addresses both large- and small-scale behaviors.
Note that the condition that we impose on the ambient metric space should address behaviors in every scale since the intersection graph does not change when we scale the ambient space by multiplying an arbitrary nonzero constant.

We first introduce notions in order to define asymptotic dimension.
Let $(X,d)$ be an (extended) metric space\footnote{All metric spaces in this paper are extended metric spaces. That is, $d(x,y)$ is allowed to be infinite for points $x,y$ in $X$.}. 
The \emph{diameter} of a subset $U$ of $X$, denoted by $\diam(U)$, is $\sup_{x,y \in U}d(x,y)$.
For $r>0$, we say that a collection $\mc{U}$ of subsets of $X$ is \emph{$r$-disjoint} if $\inf_{u \in U, u' \in U'}d(u,u')>r$ for any distinct $U,U' \in \U$.
For a nonnegative integer $n$, a function $f:\mathbb{R}_+ \to \mathbb{R}_+$ is an \emph{$n$-dimensional control function} for $(X,d)$ if for all $r \in \mathbb{Z}_+$, there is a cover $\mc{U} = \bigcup_{i=1}^{n+1}\mc{U}_i$ of $X$ such that $\sup_{U \in \U}\diam(U) \leq f(r)$ and $\U_i$ is $r$-disjoint for each $i \in [n+1]$.
A \emph{dilation} is a function $f:\mathbb{R}_+ \to \mathbb{R}_+$ such that there exists $c>0$ with $f(x)=cx$ for every $x \in \mathbb{R}_+$.

The \emph{asymptotic dimension} of a metric space $(X,d)$ is the infimum of nonnegative integers $n$ such that there exists an $n$-dimensional control function for $(X,d)$.
The \emph{Assouad-Nagata dimension} of a metric space $(X,d)$ is the infimum of nonnegative integers $n$ such that some dilation is an $n$-dimensional control function for $(X,d)$.
Note that Assouad-Nagata dimension addresses both large-scale and small-scale behaviors.
We refer readers to \cite{BD08} for a survey about asymptotic dimension.

We study the asymptotic dimension of graph metrics.
We view a (finite or infinite) graph $G$ as a metric space $(V(G),\dist_G)$, where $\dist_G(u,v)$ is the infimum of the length of paths between $u$ and $v$ in~$G$ for any $u,v \in V(G)$, and the length of a path is the number of its edges.
The \emph{asymptotic dimension} of a graph class $\mc{G}$, denoted by $\asdim(\mc{G})$, is defined as the asymptotic dimension of the (possibly infinite) graph that is the disjoint union of all graphs in $\mc{G}$; equivalently, $\mc{G}$ has asymptotic dimension at most $n \in \mathbb{Z}_+$ if some function $f:\mathbb{R}_+\to\mathbb{R}_+$ is an $n$-dimensional control function for all graphs in $\mc{G}$.

\subsection{Space-filling families}

Our first main result addresses the question that given a positive integer $n$ and a kind of shapes in ${\mathbb R}^n$, what is the asymptotic dimension of the class of all intersection graphs of families of this kind of shapes?
A concrete example is the following:

\begin{question} \label{question:asdim_ball}
Let $n$ be a positive integer, and let $\F$ be the class of intersection graphs of families of closed balls in ${\mathbb R}^n$.
What is the asymptotic dimension of $\F$?
\end{question}

It is not hard to show that every $n$-dimensional grid is an intersection graph of a family of unit balls in ${\mathbb R}^n$.
The $n$-dimensional grids are quasi-isometric to the $n$-dimensional Euclidean space, which has asymptotic dimension $n$ \cite{G93}.
So $n$ is a lower bound for \cref{question:asdim_ball}.

We provide the following almost matching upper bound for \cref{question:asdim_ball} in this paper.

\begin{theorem} \label{ball_intro}
Let $n$ be a positive integer, and let $\F$ be the class of intersection graphs of families of closed balls in ${\mathbb R}^n$.
Then $n \leq \asdim(\F) \leq n+1$.
\end{theorem}

\cref{ball_intro} is an immediate corollary of our more general result about space-filling families.
For a function $f:\mathbb{R}_+\to\mathbb{R}_+$, a family $\mc{S}$ of subsets of a metric space $(X,d)$ is \emph{$f$-space-filling} if every member of $\Se$ has finite diameter, and for any $r,s>0$ and $x\in X$, the ball $B(x,r)$ of radius $r$ around $x$ is intersected by at most $f(r/s)$ pairwise disjoint members of $\mc{S}$ of diameter at least $s$.

It is easy to see that for every positive integer $n$, there exists a function $f$ such that every family of balls in ${\mathbb R}^n$ is $f$-space-filling by considering the volume of the balls.
So the following theorem, which is our first main result, implies \cref{ball_intro}.

\begin{theorem} \label{asdim_space_filling_intro}
For any positive integer $n$, function $f: {\mathbb R}_+ \rightarrow {\mathbb R}_+$ and metric space $(X,d)$ of Assouad-Nagata dimension at most $n$, if $\F$ is the class of the intersection graphs of $f$-space-filling families of subsets of $X$, then $\asdim(\F) \leq n+1$. 
\end{theorem}

\cref{asdim_space_filling_intro} improves the upper bound $2n+1$ proved by Dvo\v{r}\'{a}k and Norin \cite{DvoNor25} and is optimal both quantitatively and qualitatively.
Quantitatively, we prove the following lower bound matching the upper bound in \cref{asdim_space_filling_intro}.

\begin{theorem} \label{asdim_space_filling_lower_intro}
For every positive integer $n$, there exists a function $f: {\mathbb R}_+ \rightarrow {\mathbb R}_+$  such that if $\F$ is the class of the intersection graphs of finite $f$-space filling families of subsets of $\bb{R}^n$, then $\asdim(\F) \geq n+1$.
\end{theorem}

Qualitatively, the conditions of \cref{asdim_space_filling_intro} cannot be relaxed and the conclusion cannot be strengthened in the following sense:
First, as observed by Dvo\v{r}\'{a}k and Norin \cite{DvoNor25}, a certain set of expanders of infinite asymptotic dimension is the subclass of the class of intersection graphs of families of closed axis-parallel boxes in ${\mathbb R}^3$ if we allow those boxes to be very thin and long (i.e.\ each box is the Cartesian product of a small square and a long interval).
Hence the condition for being space-filling cannot be dropped.
Second, the condition that $(X,d)$ has bounded Assouad-Nagata dimension cannot be weakened to be having bounded asymptotic dimension because asymptotic dimension does not address small-scale behaviors.
In particular, our proof of \cref{asdim_space_filling_lower_intro} shows that the members of $\F$ are contained in a subspace of $\mathbb{R}^n$ of bounded diameter and hence the ambient space can be chosen to have asymptotic dimension $0$.
Third, the conclusion that $\F$ has asymptotic dimension at most $n+1$ cannot be strengthened to having bounded Assouad-Nagata dimension: Dvo\v{r}\'{a}k and Norin \cite{DvoNor25} observed that the class of the intersection graphs of balls in ${\mathbb R}^3$ has infinite Assouad-Nagata dimension.

Now we show some concrete applications of \cref{asdim_space_filling_intro} in addition to \cref{ball_intro}.

We say that a metric space $(X,d)$ is a \emph{doubling space} if there exists $C>0$ such that for any $r>0$, every ball in $(X,d)$ of radius $r$ is contained in a union of at most $C$ balls of radius $r/2$; the infimum of such $C$ is called the \emph{doubling constant}.
Lang and Schlichenmaier \cite{LS05} showed that every doubling space has finite Assouad-Nagata dimension. 
In addition, for any $\eta>0$, we say that a subset $S$ of a metric space $(X,d)$ is \emph{$\eta$-round} if $S$ has finite diameter and for any $v \in S$ and $r \leq \diam(S)$, there exists $v' \in S$ such that $B(v',\eta r) \subseteq S \cap B(v,r)$.
Dvo\v{r}\'{a}k and Norin \cite{DvoNor25} observed that for any $\eta>0$ and $C>0$, there exists $f: {\mathbb R}_+ \rightarrow {\mathbb Z}_+$ such that every family of $\eta$-round sets in a metric space with doubling constant at most $C$ is $f$-space-filling.
This together with \cref{asdim_space_filling_intro} leads to the following corollary:

\begin{corollary} \label{round_intro}
Let $n$ be a positive integer, and let $(X,d)$ be a doubling space of Assouad-Nagata dimension at most $n$.
If $\F$ is the class of intersection graphs of families of $\eta$-round subsets of $(X,d)$, then $\asdim(\F) \leq n+1$.
\end{corollary}

Round sets are related to sets of bounded aspect ratio.
Let $n$ be a positive integer.
For a subset $S$ of ${\mathbb R}^n$, we define the \emph{height} of $S$ to be the infimum of $h>0$ such that $S$ is contained between two parallel hyperplanes at distance $h$, and we define the \emph{aspect ratio} of $S$ to be the ratio of the diameter of $S$ to the height of $S$.
Thin and long boxes have large aspect ratio; recall that some expanders of infinite asymptotic dimension are intersection graphs of families of thin and long boxes in ${\mathbb R}^3$.
Dvo\v{r}\'{a}k and Norin \cite{DvoNor25} observed that every compact convex set in ${\mathbb R}^n$ of aspect ratio at most $r$ is $\frac{1}{2rn}$-round.
It is easy to see that the Assouad-Nagata dimension of ${\mathbb R}^n$ is at most $n$, so \cref{round_intro} implies the following corollary:

\begin{corollary} \label{aspect_intro}
Let $n$ be a positive integer, and let $r \geq 1$ be a real number.
Then the class of intersection graphs of families of compact convex sets of aspect ratio at most $r$ in ${\mathbb R}^n$ has asymptotic dimension at most $n+1$.
\end{corollary}

Both \cref{round_intro} and \cref{aspect_intro} improve the $2n+1$ upper bound in \cite{DvoNor25}.

\subsection{Intersection graphs of boundaries}

Some natural graph classes of intersection graphs are not defined by space-filling families.
For example, one can consider intersection graphs of spheres in ${\mathbb R}^n$. 
Spheres in ${\mathbb R}^n$ do not form a space-filling family, so \cref{asdim_space_filling_intro} cannot be applied directly.
However, every sphere is the boundary of a ball, and balls form a space-filling family.

Our second main result (\cref{t:main}) shows that the asymptotic dimension of the intersection graph of a family of closed sets of a topological space does not change when the sets are replaced by their boundaries, under certain fairly mild and natural conditions.

The following notation will be used to formally state our result. 
Let $\mc{F}$ be a family of sets. 
We define the \emph{intersection asymptotic dimension of $\mc{F}$}, denoted by $\iasdim(\mc{F})$, to be the asymptotic dimension of the class of the intersection graphs of all (finite or infinite) subfamilies of $\mc{F}$.
In particular, \cref{ball_intro} is equivalent to the statement that if $\Se$ is the family of the balls in ${\mathbb R}^n$, then $n \leq \iasdim(\Se) \leq n+1$.
For a family $\mc{F}$ of subsets of a topological space, we define $\bd\mc{F}$ to be $\{ \bd{X}: X \in \mc{F}\}$, where each $\partial X$ denotes the boundary of $X$.

\begin{theorem}\label{t:main}
Let $T$ be a connected topological space.
If $\mc{F}$ is a family of subsets of $T$ such that
	\begin{enumerate}
		\item for every $S \in \mc{F}$, $S$ is closed, and $S$ and $\bd{S}$ are non-empty and connected, and
		\item $\bigcup_{S \in \mc{S}}S \neq T$ for every finite subfamily $\mc{S}$ of $\mc{F}$,
	\end{enumerate}	
then $$\iasdim(\mc{F}) \leq \iasdim(\bd\mc{F}) \leq \max\{\iasdim(\mc{F}),1\}.$$
\end{theorem}

Note that the rightmost term $\max\{\iasdim(\mc{F}),1\}$ in \cref{t:main} cannot be replaced by $\iasdim(\mc{F})$. 
For example, if $\partial\F=\{C_i: i \in {\mathbb Z}_+\}$, where each $C_i$ is the circle in ${\mathbb R}^2$ intersecting the $x$-axis at the points $((-1)^i i,0)$ and $((-1)^{i+1}(i+1),0)$ and having the center at the $x$-axis, and $\F$ consists of the disks having members in $\partial\F$ as their boundary, then $\iasdim(\F)=0$ (since the members of $\F$ pairwise intersect), but $\iasdim(\partial\F) \geq 1$ (since the intersection graph $I(\partial \F)$ is a path of infinite length). 
In fact, it is easy to see that a class of graphs has asymptotic dimension 0 if and only if there exists $C>0$ such that every component of every graph in this class has diameter at most $C$.
So $\max\{\iasdim(\mc{F}),1\} \neq \iasdim(\F)$ only happens in a very restricted situation.

The following is an immediate corollary by the combination of \cref{ball_intro} and \cref{t:main}.

\begin{corollary} \label{sphere_intro}
Let $n \geq 2$ be a positive integer.
If $\F$ is the class of intersection graphs of families of spheres in ${\mathbb R}^n$, then $n \leq \asdim(\F) \leq n+1$.
\end{corollary}

\cref{sphere_intro} was motivated by a question proposed by Georgakopoulos in the Banff workshop ``New Perspectives in Colouring and Structure'' in 2024.
When this paper was under preparation, Davies, Georgakopoulos, Hatzel and McCarty \cite{DGHM25} proved a weaker bound $\asdim(\F) \leq 2n+2$ and conjectured that $\asdim(\F)=n$. 

In fact, we prove a mild generalization of \cref{t:main} (\cref{t:main2}) as it turns out that the proof simplifies in a more general setting. 

For a subset $S$ of a topological space, we say that $W$ is an \emph{augmentation of $S$} if  $\bd {W} \subseteq S \subseteq W$.
Let $\mc{F},\mc{F}'$ be two families of subsets of a topological space. 
An \emph{augmentation of $\mc{F}$} is a bijection $\omega: \mc{F} \to \mc{F}'$ such that $\omega(S)$ is an augmentation of $S$ for every $S \in \mc{F}$.
We say that $\mc{F}'$ is an \emph{augmentation family for $\mc{F}$} if there exists an augmentation $\omega: \mc{F} \to \mc{F}'$ of $\mc{F}$. 
We say that $\mc{F}'$ is \emph{closed} if every element of $\mc{F}'$ is closed.

\begin{theorem} \label{t:main2}
Let $T$ be a connected topological space.
Let $\mc{F}$ and $\mc{F}'$ be two families of non-empty connected subsets of $T$ such that $\mc{F}'$ is a closed augmentation family for $\mc{F}$. 
If $\bigcup_{S \in \mc{S}} S \neq T$ for every finite subfamily $\mc{S}$ of $\mc{F}'$, then $$\iasdim(\mc{F}') \leq \iasdim(\mc{F}) \leq \max\{\iasdim(\mc{F}'),1\}.$$
\end{theorem}

Clearly $\mc{F}$ is an augmentation family for $\bd\mc{F}$, so \cref{t:main2} implies \cref{t:main}.   

\cref{t:main2} implies a strengthening of \cref{sphere_intro}: if $n \geq 2$ is a positive integer, and $\F$ is a family of non-empty connected subsets of ${\mathbb R}^n$, where each member of $\F$ is obtained from a closed ball in ${\mathbb R}^n$ by removing a subset of its interior, then $n \leq \iasdim(\F) \leq n+1$, since some family $\F'$ of closed balls in ${\mathbb R}^n$ is a closed augmentation family of $\F$.

\subsection{Organization}

In \cref{sec:preliminary}, we first introduce some basic terminologies in graph theory, and then prove simple results about augmentation of sets and asymptotic dimension.
We prove \cref{t:main2} in \cref{sec:augmentation_families} by assuming a key lemma, and then prove the key lemma in \cref{sec:pf_main3}.
The next goal is to prove \cref{asdim_space_filling_intro}.
We develop machinery for tree-decomposition and weak diameter coloring of metric spaces in \cref{s:treedec} and prove \cref{asdim_space_filling_intro} in \cref{sec:asdim_space_filling}.
Finally, we prove \cref{asdim_space_filling_lower_intro} in \cref{sec:lower_bound}.

\section{Preliminaries} \label{sec:preliminary}

We first define notations that will be frequently used in this paper.
Let $G$ be a graph, and let $S$ be a subset of $V(G)$.
We denote by $G[S]$ the subgraph of $G$ induced by $S$.
We define $N_G(S)$ to be the set consisting of the vertices in $V(G)-S$ adjacent in $G$ to some vertex in $S$, and we define $N_G[S] = N_G(S) \cup S$.
For an integer $k$, we define $[k] = \{x \in {\mathbb Z}: 1 \leq x \leq k\}$.
For any function $f$ and subset $X$ of the domain of $f$, we denote by $f|_X$ the restriction of $f$ to $X$.

In the rest of this section, we introduce some simple properties for augmentations and asymptotic dimension and prove an easier part of \cref{t:main2} as a warm up.
We start by the following useful lemma relating properties of sets and their augmentations.

\begin{lemma} \label{o:union}	
Let $\mc{S}$ be a finite collection of non-empty connected subsets of a connected topological space $T$.
For every $X \in \mc{S}$, let $\omega(X)$ be a connected closed augmentation of $X$. 
Then the following statements hold.
	\begin{enumerate}
		\item $\bigcup_{X \in \mc{S}}\omega(X)$ is an augmentation of $\bigcup_{X \in \mc{S}}X$.
		\item For any $X,Y \in \mc{S}$, if $X \cap Y \neq \emptyset$, then $\omega(X) \cap \omega(Y) \neq \emptyset$.
		\item If $I(\mc{S})$ is connected, then $\bigcup_{X \in \mc{S}}X$ and $\bigcup_{X \in \mc{S}}\omega(X)$ are connected.
		\item For any $X,Y \in \mc{S}$, if $X \cap Y = \emptyset$, then 
			\begin{enumerate}
				\item $\omega(X) \cap \omega(Y) = \emptyset$, or 
				\item $\omega(X) \cup \omega(Y) = T$, or
				\item $\omega (X) \subsetneqq \omega(Y)$ and $Y \cap \omega(X)=\emptyset$, or
				\item $\omega (Y) \subsetneqq \omega(X)$ and $X \cap \omega(Y)=\emptyset$.
			\end{enumerate}
	\end{enumerate}	 
\end{lemma}

\begin{proof}
Since $\Se$ is finite, we know that $\bigcup_{X \in \mc{S}}\omega(X)$ is closed.
Hence $$\bd{\s{\bigcup_{X \in \mc{S}}\omega(X)}} \subseteq \bigcup_{X \in \mc{S}}\bd(\omega(X)) \subseteq  \bigcup_{X \in \mc{S}}X \subseteq \bigcup_{X \in \mc{S}}\omega(X),$$
so Statement 1 holds.
	 
Statement 2 clearly holds since $\omega(X) \cap \omega(Y)  \supseteq X \cap Y \neq \emptyset$ for any $X \cap Y \neq \emptyset$.	
	 
Now we show Statement 3.
Note that the union of two connected sets with non-empty intersection is connected.
Since $\Se$ is finite, if $I(\mc{S})$ is connected, then $\bigcup_{X \in \mc{S}}X$ and $\bigcup_{X \in \mc{S}}\omega(X)$ are connected by Statement 2.

Finally, we prove Statement 4.

Let $X,Y \in \mc{S}$ with $X \cap Y = \emptyset$.
Let $X'=\omega(X)$ and $Y'=\omega(Y)$ for brevity.  
As $X \cap Y = \emptyset$ and $\partial Y' \subseteq Y$, we have $X \cap \partial Y'= \emptyset$, so $X \cap Y' = X \cap \brm{int} Y'$. 
Since $X$ is connected, it follows that $X \cap Y' = \emptyset$ or $X \subseteq \brm{int} Y'$. 
Indeed, otherwise, $X -Y'$ and $X \cap Y' = X \cap \brm{int} Y'$ are two non-empty relatively open sets in $X$ with the union $X$, contradicting connectivity of $X$. 
Symmetrically, we have $Y \cap X' = \emptyset$ or $Y \subseteq \brm{int} X'$.
	 
We first assume that $X \subseteq \brm{int}Y'$ and $Y \subseteq  \brm{int} X'$.
Then $$\bd(X' \cup Y') \subseteq \bd X' \cup \bd Y' \subseteq X \cup Y \subseteq \brm{int}(X' \cup Y').$$
It implies that $X' \cup Y'$ is open.
Since $X'$ and $Y'$ are closed, $X' \cup Y'$ is closed. 
As $T$ is connected it follows that $X' \cup Y' = T$, implying that the outcome (b) holds. 
	 
Thus we may assume without loss of generality that $X \not \subseteq \brm{int}Y'$.
So $X \cap Y' = \emptyset$. 
Hence $(\partial X') \cap Y'\subseteq X \cap Y'=\emptyset$.
So $Y' \cap X' = Y' \cap \brm{int} X'$.
If $X' \cap Y' \neq \emptyset$ and $Y' \not \subseteq X'$, then $Y' \cap X' = Y' \cap \brm{int} X'$ and $Y'-X'$ is a cover of $Y'$ by two disjoint non-empty relatively open sets, contradicting connectivity of $Y'$.
Hence $X' \cap Y' = \emptyset$ or $Y' \subseteq X'$.
The former implies (a).
The latter implies (d) unless $X'=Y'$.
So we may assume $X'=Y'$.
Then $\partial X'=\emptyset$ since $X'=Y'$ is closed and $(\partial X') \cap Y'=\emptyset$.
So $X'$ is both open and closed.
Hence $X'=T$ or $X'=\emptyset$ since $T$ is connected. 
Since $X'=Y'$, (a) or (b) holds.
\end{proof}

Next we state and prove some preliminary results for asymptotic dimension.
The following special case of \cite[Theorem A.2]{BBEGLPS23} shows that it suffices to study finite graphs when studying the asymptotic dimension of finite or infinite graphs.

\begin{theorem}[{\cite[Theorem A.2]{BBEGLPS23}}] \label{finite_asdim}
For every function $f: \mathbb{Z}_+ \rightarrow \mathbb{Z}_+$, there exists a function $g: \mathbb{Z}_+ \rightarrow \mathbb{Z}_+$ such that for any $n \in \mathbb{Z}_+$ and finite or infinite graph $G$, if $f$ is an $n$-dimensional control function of the set of all finite induced subgraphs of $G$, then $g$ is an $n$-dimensional control function of $G$.
\end{theorem}

Based on \cref{finite_asdim}, we introduce the restricted version of $\iasdim$ that only considers intersection graphs of finite subfamilies.
For a family $\mc{F}$ of sets, we define $\iasdim_{<\infty}(\mc{F})$ to be the asymptotic dimension of the family of the intersection graphs of all finite subfamilies of $\mc{F}$.

\begin{lemma} \label{finite_iasdim}
If $\F$ is a family of sets, then $\iasdim(\F)=\iasdim_{<\infty}(\F)$.
\end{lemma}

\begin{proof}
Let $\A$ be the family of the intersection graphs of all subfamilies of $\F$.
Let $\B$ be the family of the intersection graphs of all finite subfamilies of $\F$.
Then $\iasdim(\F)=\asdim(\A)$ and $\iasdim_{<\infty}(\F)=\asdim(\B)$.
Hence it suffices to show $\asdim(\A)=\asdim(\B)$.

Clearly, $\B \subseteq \A$.
So $\asdim(\B) \leq \asdim(\A)$.

Now we show $\asdim(\A) \leq \asdim(\B)$.
We may assume $\asdim(\B)<\infty$, for otherwise we are done.
So there exists an $\asdim(\B)$-dimensional control function $f$ of $\B$.
Let $g$ be the function that only depends on $f$ as stated in \cref{finite_asdim}.
For every $G \in \A$, every finite induced subgraph of $G$ is an intersection graph of a finite subfamily of $\F$ and hence is in $\B$, so $f$ is an $\asdim(\B)$-dimensional control function of the set of all finite induced subgraphs of $G$, implying that $g$ is an $\asdim(\B)$-dimensional control function of $G$ by \cref{finite_asdim}.
Since $g$ is independent from $G$, we know that $g$ is an $\asdim(\B)$-dimensional control function of $\A$.
So $\asdim(\A) \leq \asdim(\B)$.
\end{proof}

It is convenient to work with asymptotic dimension of graphs by using an equivalent definition stated in terms of weak diameter coloring.
A \emph{coloring} of a graph $G$ is a function with domain $V(G)$; for a positive integer $k$, a \emph{$k$-coloring} is a coloring whose codomain is $[k]$. 
Given a coloring $c$ of a graph $G$, we say that a subgraph $H$ of $G$ is \emph{$c$-monochromatic} if all its vertices receive the same color in $c$; a \emph{$c$-monochromatic component} in $G$ is a maximal $c$-monochromatic connected subgraph of $G$.
The \emph{weak diameter in $G$} of a subgraph $H$ of $G$ is $\sup_{u,v \in V(H)}\dist_G(u,v)$. 
The \emph{weak diameter in $G$} of a coloring $c$ of $G$ is defined to be the supremum of the weak diameter of a $c$-monochromatic component of $G$.

For a positive integer $\ell$, the {\it $\ell$-th power of $G$}, denoted by $G^\ell$, is the graph with $V(G^\ell)=V(G)$ such that two distinct vertices $u,v$ are adjacent in $G^\ell$ if and only if $\dist_G(u,v) \leq \ell$. The following lemma establishes a connection between the asymptotic dimension of a graph class and the weak diameter of colorings of powers of graphs in the class.

\begin{lemma}[{\cite[Proposition 1.17]{BBEGLPS23}}] \label{l:asweak}
Let $k$ be a positive integer.
Then a class $\mc{G}$ of graphs has asymptotic dimension at most $k-1$ if and only if there exists a function $f:\mathbb{Z}_+\to\mathbb{Z}_+$ such that for every positive integer $\ell$ and for every $G\in\mc{G}$, the graph $G^{\ell}$ has a $k$-coloring of weak diameter in $G^\ell$ at most $f(\ell)$.
\end{lemma}

Now we prove the easier part of \cref{t:main2}, with a slightly weaker assumption.

\begin{lemma} \label{easy_main2}
Let $T$ be a connected topological space.
Let $\mc{F}$ and $\mc{F}'$ be two families of non-empty connected subsets of $T$ such that $\mc{F}'$ is a closed augmentation family for $\mc{F}$. 
If $\bigcup_{S \in \mc{S}} S \neq T$ for every subfamily $\mc{S}$ of $\mc{F}'$ with $|\mc{S}| \leq 2$, then $\iasdim(\mc{F}') \leq \iasdim(\mc{F})$.
\end{lemma}

\begin{proof}
By \cref{finite_iasdim}, it suffices to show $\iasdim_{<\infty}(\mc{F}') \leq \iasdim_{<\infty}(\mc{F})$.
We may assume $\iasdim_{<\infty}(\mc{F})<\infty$.
Let $k=\iasdim_{<\infty}(\mc{F})+1$.
By \cref{l:asweak}, there exists a function $f:\mathbb{Z}_+\to\mathbb{Z}_+$ such that for every positive integer $\ell$ and for every finite $\Se \subseteq \F$, the graph $(I(\Se))^{\ell}$ has a $k$-coloring of weak diameter in $(I(\Se))^{\ell}$ at most $f(\ell)$.
To prove this lemma, it suffices to show that for every positive integer $\ell$ and for every finite $\Se \subseteq \F'$, the graph $(I(\Se))^{\ell}$ has a $k$-coloring of weak diameter in $(I(\Se))^{\ell}$ at most $\ell \cdot f(\ell)+2$.

Let $\ell$ be a positive integer, and let $\Se$ be a finite subfamily of $\F'$.
Let $G=I(\Se)$.
It suffices to show that $G^\ell$ has a $k$-coloring of weak diameter in $G^{\ell}$ at most $\ell \cdot f(\ell)+2$.

Let $\M$ be the set of $\subseteq$-maximal elements of $\Se$.
Note that $\M \subseteq \Se \subseteq \F'$.
Let $\omega: \F \rightarrow \F'$ be an augmentation of $\F$.
Let $H$ be the intersection graph of the family $\{\omega^{-1}(X): X \in \M\}$.
Since $\{\omega^{-1}(X): X \in \M\}$ is a finite subfamily of $\F$, there exists a $k$-coloring $c_H$ of $H^\ell$ of weak diameter in $H^\ell$ at most $f(\ell)$.

Since $\M$ is the set of $\subseteq$-maximal elements of $\Se$, there exists a function $\iota: \Se \rightarrow \M$ such that $X \subseteq \iota(X)$ for every $X \in \Se$.
Define $c$ to be the $k$-coloring of $G$ such that $c(X) = c_H(\omega^{-1}(\iota(X)))$ for every $X \in \Se$.
It suffices to show that the weak diameter of $c$ in $G^\ell$ is at most $\ell \cdot f(\ell)+2$.

Suppose to the contrary.
Then there exist a positive integer $t$ and a sequence $X_1,X_2,...,X_t$ of elements of $\Se$ with $c(X_1)=c(X_2)=...=c(X_t)$ such that $\dist_{G^\ell}(X_1,X_t)>\ell \cdot f(\ell)+2$ but $\dist_G(X_j,X_{j+1}) \leq \ell$ for every $j \in [t-1]$.
Hence $c_H(\omega^{-1}(\iota(X_i)))$ is identical for all $i \in [t]$.

\medskip

\noindent{\bf Claim 1:} For every $i \in [t-1]$, $\dist_H(\omega^{-1}(\iota(X_i)),\omega^{-1}(\iota(X_{i+1}))) \leq \ell$.

\noindent{\bf Proof of Claim 1:}
Fix $i \in [t-1]$.
Since $\dist_G(X_i,X_{i+1}) \leq \ell$, there exist $s \in [\ell+1]$ and a path $Y_1Y_2...Y_s$ in $G$ such that $Y_1=X_i$ and $Y_s=X_{i+1}$.
Let $j \in [s-1]$.
Since $Y_j \cap Y_{j+1} \neq \emptyset$, we have $\iota(Y_j) \cap \iota(Y_{j+1}) \supseteq Y_j \cap Y_{j+1} \neq \emptyset$.
By the assumption of this lemma, $\iota(Y_j) \cup \iota(Y_{j+1}) \neq T$ since $\{\iota(Y_j),\iota(Y_{j+1})\}$ is a subfamily of $\F'$.
Moreover, none of $\iota(Y_j)$ and $\iota(Y_{j+1})$ is a proper subset of the other since they are $\subseteq$-maximal elements of $\Se$.
So $\omega^{-1}(\iota(Y_j)) \cap \omega^{-1}(\iota(Y_{j+1})) \neq \emptyset$ by Statement 4 in \cref{o:union}.
Hence $\dist_H(\omega^{-1}(\iota(X_i)),\omega^{-1}(\iota(X_{i+1}))) \leq \ell$.
$\Box$

\medskip

By Claim 1, $\omega^{-1}(\iota(X_1))$ and $\omega^{-1}(\iota(X_t))$ are contained in the same $c_H$-monochromatic component of $H^\ell$.
So $\dist_{H^\ell}(\omega^{-1}(\iota(X_1)),\omega^{-1}(\iota(X_t))) \leq f(\ell)$.
Hence $\dist_{H}(\omega^{-1}(\iota(X_1)), \allowbreak \omega^{-1}(\iota(X_t))) \leq \ell \cdot f(\ell)$.

Note that for any $A,B \in V(H)=\{\omega^{-1}(X): X \in \M\}$, if $A \cap B \neq \emptyset$, then $\omega(A) \cap \omega(B) \supseteq A \cap B \neq \emptyset$.
Hence $$\dist_{G^\ell}(\iota(X_1),\iota(X_t)) \leq \dist_G(\iota(X_1),\iota(X_t)) \leq \dist_{H}(\omega^{-1}(\iota(X_1)),\omega^{-1}(\iota(X_t))) \allowbreak \leq \ell \cdot f(\ell).$$
Since $\iota(X_1) \supseteq X_1$ and $\iota(X_t) \supseteq X_t$, we have $\dist_{G^\ell}(X_1,X_t) \leq \dist_{G^\ell}(\iota(X_1),\iota(X_t))+2 \leq \ell \cdot f(\ell)+2$, a contradiction.
\end{proof}

\section{Augmentation families} \label{sec:augmentation_families}

The goal of this section is to prove the remaining part of \cref{t:main2}.

For a graph $G$, we define the \emph{diameter} of $G$ to be $\sup_{u,v \in V(G)}\dist_G(u,v)$ and denote it by $\diam(G)$.
 
For $r>0$ and a family $\mc{F}$ of sets, let $\bigcup^{r}\mc{F}$ denote the family $$\{\bigcup_{S \in \mc{S}}S: \emptyset \neq \mc{S} \subseteq \mc{F}, |\mc{S}|<\infty, \diam(I(\mc{S})) \leq r\}.$$ 
Note that $\diam(I(\mc{S})) \leq r$ implies that $I(\mc{S})$ is connected. 
For each $X \in \bigcup^{r}\mc{F}$, let $\Phi^{r,\mc{\F}}(X) = \mc{S}$ for some $\mc{S} \subseteq \mc{F}$ such that $X = \bigcup_{S \in \mc{S}}S$ and $\mc{S}$ satisfies the properties above showing that $X \in \bigcup^r \mc{F}$. 
 
\begin{lemma} \label{o:union2} 
Let $T$ be a connected topological space.
If $\mc{F}$ and $\mc{F}'$ are two families of non-empty connected subsets of $T$ such that $\mc{F}'$ is a closed augmentation family for $\mc{F}$, then for every $r>0$, the families $\bigcup^{r}\mc{F}$ and $\bigcup^{r}\mc{F}'$ also consist of some non-empty connected subsets of $T$, and some augmentation family for $\bigcup^{r}\mc{F}$ is a subfamily of $\bigcup^{r}\mc{F}'$.
Moreover, for every $\mc{S} \in \{\mc{F},\mc{F}'\}$, if every element of $\mc{S}$ is closed, then every element of $\bigcup^r\mc{S}$ is closed.
\end{lemma}	
 
\begin{proof} 
Let $\omega: \mc{F} \to \mc{F}'$ be an augmentation of $\mc{F}$. 
Clearly the elements of $\bigcup^{r}\mc{F}$ and $\bigcup^{r}\mc{F}'$ are non-empty, and they are connected by Statement 3 of \cref{o:union}. 
For every $X \in \bigcup^r \mc{F}$, let $\omega'(X)= \bigcup_{S \in \Phi^{r,\mc{\F}}(X)}\omega(S)$.
Then $\omega'(X)$ is an augmentation of $X$ for every $X \in \bigcup^r \mc{F}$ by Statement 1 of \cref{o:union}.
So $\{\omega'(X): X \in \bigcup^r \mc{F}\}$ is an augmentation family for $\bigcup^r \mc{F}$.
In addition, $\omega'(X) \in \bigcup^{r}\mc{F}'$ for every $X \in \bigcup^r \F$ since $\diam(I(\{\omega(S): S \in \Phi^{r,\mc{\F}}(X)\})) \leq \diam(I(\Phi^{r,\mc{\F}}(X))) \leq r$.
Hence $\{\omega'(X): X \in \bigcup^r \mc{F}\}$ is a subfamily of $\bigcup^{r}\mc{F}'$.

Moreover, for every $\mc{S} \in \{\mc{F},\mc{F}'\}$, if every element of $\mc{S}$ is closed, then every element of $\bigcup^r\mc{S}$ is a union of finitely many closed sets and hence is closed.
\end{proof}

The main ingredient of the proof of \cref{t:main2} is the following lemma, which will be proved in \cref{sec:pf_main3}.

\begin{restatable}{lemma}{MainLemma}\label{t:main3}  
Let $T$ be a connected topological space, and let $\mc{S}$ and $\mc{S}'$ be two finite families of non-empty connected subsets of $T$ such that $\mc{S}'$ is a closed augmentation family for $\mc{S}$ and $\bigcup_{S \in \mc{S}'} S \neq T$. 
Let $k \geq 2$ be an integer.
If for every $\mc{S}'' \subseteq \mc{S}'$, there exists a $k$-coloring of $I(\mc{S}'')$ of weak diameter in $I(\mc{S}'')$ at most $r$, then there exists a $k$-coloring of $I(\mc{S})$ of weak diameter in $I(\mc{S})$ at most $2r+32$.
\end{restatable}	

Now we show that \cref{t:main3} implies \cref{t:main2}.

\begin{proof}[Proof of \cref{t:main2} assuming \cref{t:main3}]
By \cref{easy_main2}, it suffices to show $\iasdim(\F) \allowbreak \leq \max\{\iasdim(\mc{F}'),1\}$.
It suffices to show $\iasdim_{<\infty}(\F) \leq \max\{\iasdim_{<\infty}(\mc{F}'),1\}$ by \cref{finite_iasdim}.
We may assume $\iasdim_{<\infty}(\mc{F}')<\infty$, for otherwise we are done.

Let $k=\max\{\iasdim_{<\infty}(\mc{F}'),1\}+1$. 
By \cref{l:asweak}, there exists $f:\mathbb{Z}_+\to\mathbb{Z}_+$ such that 
	\begin{itemize}
		\item[(i)] for every positive integer $\ell$ and for every finite $\mc{S}'' \subseteq \mc{F}'$, there exists a $k$-coloring of $(I(\mc{S}''))^\ell$ with weak diameter in $(I(\mc{S}''))^\ell$ at most $f(\ell)$. 
	\end{itemize}
Define $g:\mathbb{Z}_+\to\mathbb{Z}_+$ to be the function such that $g(\ell) = 6(4\ell+1) \cdot f(4\ell+1)+104$ for every positive integer $\ell$. 

To show \cref{t:main2}, \cref{l:asweak} implies that it suffices to show that for any finite subfamily $\Se_0$ of $\mc{F}$ and positive integer $\ell$, the graph $(I(\mc{S}_0))^{\ell}$ has a $k$-coloring of weak diameter in $(I(\mc{S}_0))^{\ell}$ at most $g(\ell)$.

Let $\Se_0$ be a finite subfamily of $\mc{F}$, and let $\ell \in \mathbb{Z}_+$.
We shall show that the graph $(I(\mc{S}_0))^{\ell}$ has a $k$-coloring of weak diameter in $(I(\mc{S}_0))^{\ell}$ at most $g(\ell)$.

Let $\omega: \mc{F} \rightarrow \mc{F'}$ be an augmentation of $\mc{F}$.
Let $\omega(\Se_0) = \{\omega(S): S \in \mc{S}_0\}$.
Then $\omega(\Se_0)$ is a closed augmentation family for $\Se_0$.
Let $\mc{S} = \bigcup^{2\ell}\mc{S}_0$ and $\mc{S'} = \bigcup^{2\ell}\omega(\mc{S}_0)$. 
By \cref{o:union2}, $\Se$ and $\Se'$ consist of some non-empty connected subsets of $T$, and there exists an augmentation family $\Se''$ for $\mc{S}$ such that $\Se'' \subseteq \Se'$.
Since every member of $\F'$ is closed, every element of $\omega(\Se_0)$ is closed.
So \cref{o:union2} implies that every member of $\Se'$ is closed.
Hence every member of $\Se''$ is closed.
That is,
	\begin{itemize}
		\item[(ii)] $\Se''$ is a closed augmentation family for $\Se$.
	\end{itemize}

We will apply \cref{t:main3} to $\mc{S}$ and $\mc{S}''$ to show that $I(\mc{S})$ has a $k$-coloring with bounded weak diameter. 
First, we need to establish the necessary conditions.

\medskip

\noindent{\bf Claim 1:} For every $\M \subseteq \Se'$, there exists a $k$-coloring of $I(\M)$ of weak diameter in $I(\M)$ at most $(4\ell+1) \cdot f(4\ell+1)+1$.

\noindent{\bf Proof of Claim 1:}
Let $\M \subseteq \Se'$, and let $H = I(\M)$.
Let $\M_*= \bigcup_{X \in \M} \Phi^{2\ell,\omega(\Se_0)}(X)$, and let $H_*=I(\M_*)$.

Since $\Se_0$ is finite, $\Se'$ is finite and hence $\M$ is finite.
So $\M_*$ is a finite subfamily of $\omega(\Se_0) \subseteq \mc{F}'$.
Hence (i) implies that $(H_*)^{4\ell+1}$ has a $k$-coloring $c_*$ of weak diameter in $(H_*)^{4\ell+1}$ at most $f(4\ell+1)$. 

For each $X \in \M$, let $\phi(X)$ be a member of $\Phi^{2\ell,\omega(\Se_0)}(X) \subseteq V(H_*)$, and let $c_H(X)=c_*(\phi(X))$. 
Note that if two members $X,Y$ of $\M$ are adjacent in $H$, then $\phi(X)$ and $\phi(Y)$ are equal or adjacent in $(H_*)^{4\ell+1}$.
On the other hand, for any $X,Y \in \M=V(H)$, $$\dist_{H}(X,Y) \leq \dist_{H_*}(\phi(X), \phi(Y)) + 1 \leq (4\ell+1) \cdot \dist_{(H_*)^{4\ell+1}}(\phi(X), \phi(Y))+1.$$  
Thus for any $t \in {\mathbb Z}_+$, $\phi$ maps the vertex set of each $c_H$-monochromatic component of $H$ with weak diameter in $H$ equal to $t$ to a subset of the vertex set of a $c_*$-monochromatic component of $(H_*)^{4\ell+1}$ with weak diameter in $(H_*)^{4\ell+1}$ at least $(t-1)/(4\ell+1)$. 
It follows that the weak diameter in $H$ of $c_H$ is at most $(4\ell+1) \cdot f(4\ell+1)+1$.
$\Box$

\medskip

Since $\Se'$ is finite and every member of $\Se'$ is a union of finitely many members of $\F'$, we know that $\bigcup_{S \in \Se'}S$ is also a union of finitely many members of $\F'$.
Hence $\bigcup_{S \in \Se''}S \subseteq \bigcup_{S \in \Se'}S$ is a union of finitely many members of $\F'$.
So 
	\begin{itemize}
		\item[(iii)] $\bigcup_{S \in \Se''}S \neq T$ 
	\end{itemize}
by assumption of this theorem.

Let $R=2 \cdot ((4\ell+1) \cdot f(4\ell+1)+1)+32$.
By Claim 1, for every $\M \subseteq \Se'' \subseteq \Se'$, there exists a $k$-coloring of $I(\M)$ of weak diameter in $I(\M)$ at most $(4\ell+1) \cdot f(4\ell+1)+1$.
By (ii) and (iii), we can apply \cref{t:main3} (taking $(\Se,\Se')=(\Se,\Se'')$).
So there exists a $k$-coloring $c$ of $I(\Se)$ of weak diameter in $I(\Se)$ at most $2 \cdot ((4\ell+1) \cdot f(4\ell+1)+1)+32=R$.

Let $G_0=I(\mc{S}_0)$.
For each $X \in \mc{S}_0$, let $\psi(X)$ be the union of all $Y \in \mc{S}_0$ such that $\dist_{G_0}(X,Y) \leq \ell$, so $\psi(X) \in \mc{S}$ and we define $c_0(X)=c(\psi(X))$. 
Note that if $X,Y \in V(G_0)=\Se_0$ are adjacent in $(G_0)^{\ell}$, then $\psi(X)$ and $\psi(Y)$ are equal or adjacent in $I(\Se)$. 
On the other hand, for any $X,Y \in V(G_0)$, $$\dist_{(G_0)^{\ell}}(X,Y) \leq 3 \cdot \dist_{I(\Se)}(\psi(X),\psi(Y))+2.$$ 
It follows that for any $t \in \mathbb{Z}_+$, $\psi$ maps the vertex set of every $c_0$-monochromatic component of $(G_0)^{\ell}$ with weak diameter in $(G_0)^{\ell}$ equal to $t$ to a subset of the vertex set of a $c$-monochromatic component of $I(\Se)$ with weak diameter in $I(\Se)$ at least $(t-2)/3$. 
Thus the weak diameter of the coloring $c_0$ of $(G_0)^{\ell}$ is at most $3R+2=6(4\ell+1) \cdot f(4\ell+1)+104= g(\ell)$, as desired.
\end{proof}	
 	  
\section{Proof of \cref{t:main3}} \label{sec:pf_main3}

We prove \cref{t:main3} in this section.
We restate it for convenience of the reader.

\MainLemma*

\begin{proof}
Let $T,\mc{S},\mc{S}'$ be as stated in the lemma.
Let  $\omega: \mc{S} \to \mc{S}'$ be an augmentation of $\mc{S}$.
For every subset $\mc{S}''$ of $\mc{S}$, a \emph{dominator for $\mc{S}''$} is an element $D$ of $\mc{S}''$ such that for every $X \in \mc{S}''$, either $D \cap X \neq \emptyset$ or $\omega(X) \subseteq \omega(D)$. 
Let $G=I(\mc{S})$. 
Let $R=2r+32$.
 	
Let $c_0$ be a $k$-coloring of a subset $\mc{Z}$ of $\mc{S}$ such that either $\mc{Z} = \emptyset$, or there exists a dominator $D \in \mc{S}$ such that $\dist_G(X,D) \leq 5$ for every $X \in \mc{Z}$. 

We will show by induction on $|\mc{S}-\mc{Z}|$ that $c_0$ extends to a $k$-coloring of $G$ of weak diameter in $G$ at most $R$.
Note that it would imply this lemma because this lemma is the special case $\mc{Z}=\emptyset$.

The base case is trivial since if $\mc{S}=\mc{Z}$, then the diameter of $G$ is at most $10$.
So we may assume $\mc{Z} \subset \mc{S}$.

By extending $c_0$, we may without loss of generality assume that 
	\begin{itemize}
		\item[(i)] if $\mc{Z} \neq \emptyset$, then $\mc{Z}=\{X \in \mc{S} : \dist_G(X,D) \leq 5\}$ for some dominator $D$ for $\mc{S}$.
	\end{itemize}

For every $\F \subseteq \Se$, let $\omega(\F) = \{\omega(X): X \in \F\}$.
Let $\mc{U}$ be a subset of $\mc{S} - \mc{Z}$ such that $\omega(\mc{U})$ is the set of $\subseteq$-maximal elements of $\omega(\mc{S} - \mc{Z})$. 

\medskip

\noindent{\bf Claim 1:} There exists a $k$-coloring $c_1$ of $I(\mc{U})$ with weak diameter in $I(\mc{U})$ at most $r$.
 	
\noindent{\bf Proof of Claim 1:} 
By Statement 2 in \cref{o:union}, if $X,Y \in \mc{U}$ with $XY \in E(I(\mc{U}))$, then $\omega(X)\omega(Y) \in E(I(\omega(\mc{U})))$.
For any distinct $X,Y \in \mc{U}$ with $XY \not \in E(I(\mc{U}))$, since $\omega(X) \cup \omega(Y) \neq T$ by the assumption of this lemma and none of $\omega(X)$ and $\omega(Y)$ is a proper subset of the other by the definition of $\mc{U}$, Statement 4 in \cref{o:union} implies that $\omega(X)\omega(Y) \not \in E(I(\omega(\mc{U})))$.
Hence the graph $I(\mc{U})$ is isomorphic to $I(\omega(\mc{U}))$. 
Then the claim follows from the assumption of this lemma for $\mc{S}'' = \omega(\mc{U})$. 
$\Box$ 

\medskip
 	
Let $\preceq$ be a linear ordering of $\mc{U}$.
By the definition of $\mc{U}$, for every $X \in \mc{S} - \mc{Z}$, $\omega(X)$ is a subset of $\omega(U)$ for some element $U \in \mc{U}$.
So there exists a partition $ \{\mc{S}_U: U \in \mc{U}\}$ of  $\mc{S} - \mc{Z}$ such that $U$ is the $\preceq$-smallest element satisfying $\omega(X) \subseteq \omega(U)$ for every $X \in \mc{S}_U$. In particular, $U \in \mc{S}_U$ for every $U \in \mc{U}$.
 	
For $U \in \mc{U}$, let $\mc{F}_U = N_G[\mc{S}_U]$ and let $G_U = I(\mc{F}_U)$.

\medskip

\noindent{\bf Claim 2:} For any $U \in \mc{U}$ and $X \in \mc{F}_U - \mc{S}_U$ with $X \cap U = \emptyset$, we have $\omega(X) \subsetneqq \omega(U)$ and $\dist_{G_U}(X, U) \leq 3$. 

\noindent{\bf Proof of Claim 2:}
Since $X \in \F_U-\Se_U$, there exists $Y \in \mc{S}_U$ such that $X \cap Y \neq \emptyset$. 
As $Y \subseteq \omega(Y) \subseteq \omega(U)$ it follows that $X \cap \omega (U) \supseteq X \cap Y \neq \emptyset$. 
So $\omega(X)  \subsetneqq \omega(U)$ and $U \cap \omega (X)=\emptyset$ by Statement 4 in \cref{o:union}. 

We first suppose $X \in \mc{Z}$. 
By (i), since $X \in \mc{Z}$, there exists a dominator $D$ for $\mc{S}$ with $\dist_G(X,D) \leq 5$.
Choose $X' \in \mc{S}$ such that $\omega(X') \subseteq \omega(U)$ and $\dist_G(X',D)$ is minimum. 
Note that $X$ is a candidate for $X'$, so $X'$ exists and $\dist_G(X',D) \leq \dist_G(X,D) \leq 5$.
Since $U \in \mc{U}$, we know $U \not \in \mc{Z}$, so (i) implies that $U \cap D = \emptyset$.
Since $D$ is a dominator for $\Se$ and $U \cap D = \emptyset$, we have $\omega(U) \subseteq \omega(D)$.
If $X'=D$, then $\omega(D)=\omega(X') \subseteq \omega(U) \subseteq \omega(D)$, implying $\omega(U)=\omega(D)$, and hence $U \cap D \supseteq \bd\omega(U) \cap \bd\omega(D)=\bd\omega(D)$; but $\partial\omega(D) \neq \emptyset$ since $\omega(D)$ is a closed subset of a connected topological space $T$ and $\emptyset \neq \omega(D) \neq T$, a contradiction. 
Thus $X' \neq D$.
Since $1 \leq \dist_G(X',D) \leq 5$ is finite, there exists $X'' \in \Se$ such that $X' \cap X'' \neq \emptyset$ and $\dist_G(X'', D) = \dist_G(X',D)-1\leq 4$. 
Note that $X'' \in \mc{Z}$ by (i), and we know $\omega(X'') \not \subseteq \omega(U)$ by the choice of $X'$. 
Moreover, $\omega(X'') \cap \omega(U) \supseteq X'' \cap X' \neq \emptyset$. 
If $X'' \cap U \neq \emptyset$, then $\dist_G(U,D) \leq \dist_G(X'',D)+1 \leq 5$, contradicting $U \not \in \mc{Z}$. 
So $X'' \cap U = \emptyset$.
Hence $X'' \cap \omega(U)=\emptyset$ by Statement 4 in \cref{o:union}, again a contradiction, as $X'' \cap \omega(U) \supseteq X'' \cap X' \neq \emptyset$.

Therefore, $X \in \Se-\mc{Z}$.
This together with $X \not \in \Se_U$ imply that there exists $V \in \mc{U}$ with $V \precneqq U$ such that $\omega(X) \subseteq \omega(V)$. 
Since $\omega(X) \subseteq \omega(U)$, it follows that $\omega(V) \cap \omega(U) \neq \emptyset$.
Since $U$ and $V$ are distinct elements of $\mc{U}$, Statement 4 in \cref{o:union} implies $U \cap V \neq\emptyset$. 
In particular, $V \in \mc{F}_U$. 
Further, $\omega(Y) \cap \omega(V) \supseteq \omega(Y) \cap \omega(X) \supseteq Y \cap X \neq \emptyset$. 
As $Y \in \mc{S}_U$ and $V \precneqq U$, we have $\omega(Y) \not \subseteq \omega(V)$. 
Then Statement 4 in \cref{o:union} implies $Y \cap V \neq \emptyset$.
So $XYVU$ is a path of length at most three from $X$ to $U$ in $G_U$. 
Hence $\dist_{G_U}(X, U) \leq 3$.
$\Box$ 

\medskip

Let $c_2$ be the partial $k$-coloring of $G$  defined as follows
\begin{itemize}
	\item $c_2(X)=c_0(X)$ for every $X \in \mc{Z}$.
	\item $c_2(X)=c_1(U)$ for any $U \in \mc{U}$ and $X \in \mc{S}_U$ such that $\dist_{G_U}(X,U) \leq 4$.
	\item $c_2(X) = \min([k] - \{c_1(U)\})$ for any $U \in \mc{U}$ and $X \in \mc{S}_U$ such that $\dist_{G_U}(X,U) = 5$.
\end{itemize}
For every $U \in \mc{U}$, let $\mc{Z}_U = \{X \in \mc{F}_U: c_2(X)$ is defined$\}$.

By Claim 2, 
	\begin{itemize}
		\item[(ii)] for every $U \in \mc{U}$, we know that $U$ is a dominator for $\mc{F}_U$, and $\dist_{G_U}(X,U) \leq 3$ for every $X \in \mc{Z}_U - \Se_U$. 
	\end{itemize}
For every $U \in \mc{U}$, the definition of $\mc{Z}_U$ and (ii) imply that $U$ is a dominator for $\mc{F}_U$ such that $\dist_{G_U}(X,U) \leq 5$ for every $X \in \mc{Z}_U$. 
Moreover, for every $U \in \mc{U}$, we have $|\mc{F}_U - \mc{Z}_U|<|\mc{S}-\mc{Z}|$ since $U \in \mc{Z}_U-\mc{Z}$ and $\mc{Z} \cap \mc{F}_U \subseteq \mc{Z}_U$.
Hence for each $U \in \mc{U}$, we can apply the inductive hypothesis to $G_U$ to extend the coloring $c_2|_{\mc{Z}_U}$ to a coloring $c_U$ of $G_U$ with weak diameter in $G_U$ at most $R$. 

Define $c$ to be the coloring of $G$ such that for every $X \in V(G)=\Se$,
	\begin{itemize}
		\item if $X \in \mc{Z}$, then $c(X)=c_2(X)$, and
		\item if $X \in \Se-\mc{Z}$, then $c(X)=c_{U_X}(X)$, where $U_X$ is the unique member of $\mc{U}$ such that $X \in \Se_{U_X}$.
	\end{itemize}
To prove this lemma, it suffices to show that $c$ has weak diameter in $G$ at most $R$. 

Suppose to the contrary, and let $\mc{Q}$ be the vertex set of a connected subgraph of $G$ with all vertices receiving the same color (say $1$) of weak diameter in $G$ larger than $R$. 

\medskip

\noindent{\bf Claim 3:} There exists $X \in \mc{Q}$ such that $X \in \mc{S}_U$ for some $U \in \mc{U}$ with $\dist_{G_U}(X, U) \geq 5$.

\noindent{\bf Proof of Claim 3:}
Since either $\mc{Z}=\emptyset$, or $\mc{Z}$ has weak diameter in $G$ at most $10$, we know that there exists $\mc{Q}' \subseteq \mc{Q} - \mc{Z}$ such that $\mc{Q}'$ is the vertex set of a connected subgraph of $G \setminus \mc{Z}$ of weak diameter in $G$ larger than $(R-12)/2 = r+10$.
Let $R'=r+10$.
Let $X_1,\ldots,X_{\ell} \in \mc{Q}'$ be the vertices of a path in $G[\mc{Q}]$ in order such that $\dist_G(X_1,X_{\ell}) \geq R'$.
For every $i \in [\ell]$, since $X_i \in \mc{Q}' \subseteq \Se-\mc{Z}$, there exists $U_i \in \U$ such that $X_i \in \Se_{U_i}$.
Since $X_i \cap X_{i+1} \neq \emptyset$ and $\omega(X_i) \subseteq \omega(U_i)$ and $\omega(X_{i+1}) \subseteq \omega(U_{i+1})$  for every $1 \leq i \leq \ell-1$, we have $\omega(U_i) \cap \omega(U_{i+1}) \neq \emptyset$, so $U_i \cap U_{i+1} \neq \emptyset$ by Statement 4 in \cref{o:union}. 

Suppose that this claim does not hold.
Then for every $i \in [\ell]$, we have $\dist_{G_{U_i}}(X_i,U_i) \leq 4$, so $c_1(U_i)=c(U_i)=c(X_i)=1$ since $c$ extends $c_2$.
Hence $U_1,\ldots,U_{\ell}$ are contained in a $c_1$-monochromatic component of $I(\mc{U})$ such that 
	\begin{align*}
		\dist_{I(\mc{U})}(U_1,U_\ell) \geq \dist_G(U_1,U_\ell) & \geq \dist_G(X_1,X_\ell)- \dist_G(U_1,X_1) - \dist_G(X_\ell,U_\ell) \\
		& \geq \dist_G(X_1,X_\ell)- \dist_{G_{U_1}}(U_1,X_1) - \dist_{G_{U_\ell}}(X_\ell,U_\ell) \\
		& \geq R'-8>r.
	\end{align*}
It contradicts Claim 1. 
$\Box$ 

\medskip

Let $X$ and $U$ be as in Claim 3.
Let $C$ be the component of $G[\mc{Q} \cap \mc{S}_U]$ containing $X$.
If $V(C)=\mc{Q}$, then $\mc{Q}$ is contained in a $c_U$-monochromatic component of $G_U$ and hence has weak diameter in $G_U$ (and hence in $G$) at most $R$, a contradiction.
So $V(C) \neq \mc{Q}$. 
In particular, there exists $Y \in (\mc{Q}-\mc{S}_U) \cap N_G(V(C)) \subseteq N_G[\Se_U] = V(G_U)$ such that $\dist_{G_U}(X,Y)<\infty$.
We choose $Y$ such that $\dist_{G_U}(X,Y)$ is minimum.
Then $\dist_{G_U}(Y,U) \leq 3$ by Claim 2. 

Hence $Y \in \{W \in \mc{Q} \cap N_G(V(C)): \dist_{G_U}(U,W) \leq 3\}$ and the set $X \in \{W \in \mc{Q} \cap V(C): \dist_{G_U}(U,W) \geq 5\}$.
Let $P$ be a path in $G[N_G[V(C)]]$ from $X$ to $Y$ with $V(P)-\{Y\} \subseteq V(C)$.
For every integer $i \geq 0$, let $A_i = \{W \in V(G_U): \dist_{G_U}(U,W)=i\}$.
Then $\{A_i: i \geq 0\}$ is a partition of $V(G_U)$ such that for every $e \in E(G_U)$, there exists $i_e \geq 0$ such that both ends of $e$ are contained in $A_{i_e} \cup A_{i_e+1}$. 
Since $P$ is a path in $G_U$ from $X \in \bigcup_{i \geq 5}A_i$ to $Y \in \bigcup_{0 \leq i \leq 3}A_i$, there exist $X_4,X_5 \in V(P)-\{Y\} \subseteq V(C) \subseteq \mc{Q} \cap  \mc{S}_U$ such that $\dist_{G_U}(X_4,U)=4$ and $\dist_{G_U}(X_5,U)=5$. 
But then $c_2(X_4) \neq c_2(X_5)$ by the definition of $c_2$.
So $c(X_4) \neq c(X_5)$, as $c$ extends $c_2$, a contradiction.
\end{proof}

\section{Tree-decompositions of graphs and metric spaces}\label{s:treedec}

Our next objective is to prove \cref{asdim_space_filling_intro}. 
In this section we develop necessary machinery related to tree-decompositions and prove a key lemma.

A \emph{tree-decomposition} of a graph $G$ is a pair $(T, \{B_t: t \in V(T)\})$ such that 
	\begin{itemize} 
		\item $B_t \subseteq V(G)$ for every $t\in V(T)$,
		\item for every $v \in V(G)$, the set $\{ t: v \in B_t \}$ induces a (non-null) subtree of $T$, and
		\item for every $uv \in E(G)$, there exists $t \in V(T)$ such that $u,v \in B_t$.
	\end{itemize}
The \emph{width} of a tree decomposition $(T, \{B_t: t \in V(T)\})$ is $\max_{t \in V(T)}|B_t| - 1 $. 
The \emph{tree-width $\tw(G)$} of a graph $G$ is the minimum width of a tree-decomposition of $G$. 

Let $(X,d)$ be a metric space. 
By analogy with the above, for every $r>0$, we define a \emph{$r$-compliant tree-decomposition} of $(X,d)$ as a pair $(T, \{B_t: t \in V(T)\})$ such that
	\begin{itemize} 
		\item $B_t \subseteq X$ for every $t\in V(T)$, 
		\item for every $x \in X$, the set $\{t: x \in B_t \}$ induces a (non-null) subtree of $T$, and
		\item for any $x,y \in X$ with $d(x,y) \leq r$, there exists $t \in V(T)$ such that $x,y \in B_t$.
	\end{itemize}

Given a finite metric space $(X,d)$ and $r \geq 0$, we naturally associate a graph $G_r(X,d)$ to $X$; formally, we define $G_r(X,d)$ to be the graph with $V(G_r(X,d))=X$ such that for any distinct $x,y \in X$, $xy \in E(G_r(X,d))$ if and only if $d(x,y) \leq r$. 
The following observation about correspondence between tree-decompositions of graphs and metric spaces follows immediately from the definitions.

\begin{obs}\label{o:graphVSspaceTD} 
Let $(X,d)$ be a finite metric space and let $r \geq 0$. 
Then a pair $(T, \mc{X})$ is an $r$-compliant tree-decomposition of $(X,d)$ if and only if $(T, \mc{X})$ is a tree-decomposition of $G_r(X,d)$.
\end{obs}

Let $(X,d)$ be a metric space.
We say that a set $W \subseteq X$ is \emph{$(k,r)$-centered} if there exist $U_1,\ldots,U_k \subseteq X$ such that $W \subseteq \bigcup_{i=1}^k U_i$ and $\diam(U_i) \leq r$ for every $1 \leq i \leq k$. 
We say that a tree-decomposition $(T, \{B_t: t \in V(T)\})$ of $(X,d)$ is \emph{$(k,r)$-centered} if $B_t$ is $(k,r)$-centered for every $t \in V(T)$.

A \emph{premetric on a set $X$} is a symmetric function $\ell : X \times X \to \bb{R}_{\geq 0} \cup \{+ \infty\}$ such that $\ell(x,x)=0$ for every $x \in X$. 
Given a premetric $\ell$ on $X$, the \emph{metric derived from $\ell$}, denoted by $d_{\ell}$, is the metric on $X$ such that $$d_{\ell}(x,y) = \inf (\ell(x,x_1)+\ell(x_1,x_2)+\ldots + \ell(x_k,y)),$$
where the infimum is taken over all finite sequences $(x_1,\ldots,x_k)$ of points in $X$.
Note that $d_\ell$ is the metric space defined by the weighted complete graph on $X$ with edge-length $\ell(x,y)$ for any distinct $x,y \in X$.

Let $(X,d)$ be a metric space, and let $J$ be a finite or infinite graph with vertex set $X$. 
We define the \emph{$(r,J)$-compression of $d$} to be the metric $d_\ell$, where $\ell$ is the premetric satisfying
\begin{equation}
	\ell(x,y) = \begin{cases} \min \{r, d(x,y)\} \qquad &\mathrm{if} \; xy \in E(J), \\
		 d(x,y) \qquad &\mathrm{if} \; xy \not \in E(J). 
	\end{cases}
\end{equation}	     

For real numbers $r$ and $D$, we say that a subset $Y$ of a metric space $(X,d)$ is \emph{$(r,D)$-controlled} if there exists a partition\footnote{In this paper, members of a partition can be empty sets.} $\mc{U}$ of $Y$ such that 
	\begin{itemize}
		\item $d(U,U') > r$ for any distinct $U, U' \in \mc{U}$, and
		\item $\diam (U') \leq D$ for every $U \in \mc{U}$.
	\end{itemize} 
We say that such a partition $\mc{U}$ is an \emph{$(r,D)$-witness for $Y$}.
We say that a coloring $c : X \to [k]$ is an \emph{$(r,D)$-controlled $k$-coloring of $(X,d)$} if $c^{-1}(\{i\})$ is $(r,D)$-controlled for every $i \in [k]$. 
We say that $(\mc{U}_1,\ldots,\mc{U}_k)$ is an \emph{$(r,D)$-witness for $c$} (or, occasionally, simply a \emph{witness}, when the parameters are clear from the context) if $\mc{U}_i$ is an $(r,D)$-witness for $c^{-1}(\{i\})$ for each $i \in [k]$.

The following is clear.  
 
\begin{obs}\label{o:graphVSspaceColor} 
Let $(X,d)$ be a finite metric space.
Let $r,N >0$ and $k \in \bb{Z}_+$.
Let $G=G_r(X,d)$. 
If $c$ is a $k$-coloring of $G$ of weak diameter in $G$ at most $N$, then $c$ is an $(r,Nr)$-controlled $k$-coloring of $(X,d)$.
\end{obs} 

Let $(X,d_X)$ and $(Y,d_Y)$ be metric spaces.
For $\alpha,\beta,r \geq 0$, we say that a map $\phi: X \to Y$ is \emph{$(\alpha,\beta,r)$-constrained} if for all $x,x' \in X$,
	\begin{itemize}	
		\item $d_Y(\phi(x),\phi(x')) \geq \alpha d_X(x,x')-\beta$, and
		\item if $d_X(x,x') \leq r$, then $d_Y(\phi(x),\phi(x')) \leq r$.
	\end{itemize} 

\begin{lemma}\label{l:constrained}
Let $\alpha >0$ and $\beta,r,D \geq 0$. 
Let $(X,d_X)$ and $(Y,d_Y)$ be finite metric spaces, and let $\phi: X \to Y$ be $(\alpha,\beta,r)$-constrained. 
Let $S$ be a subset of $Y$.
If $S$ is $(r,D)$-controlled with witness $\U$, then $\phi^{-1}(S)$ is $(r, (D+\beta)/\alpha)$-controlled with witness $\{\phi^{-1}(U): U \in \mc{U}, U \neq \emptyset\}$.
\end{lemma}

\begin{proof}
Let $\mc{U}$ be an $(r,D)$-witness for $S$. 
As $\mc{U}$ is a partition of $S$, the family $\phi^{-1}(\mc{U}) = \{\phi^{-1}(U): U \in \mc{U}, U \neq \emptyset\}$ is a partition of $\phi^{-1}(S)$.  
As $\diam(U) \leq D$ for every $U \in \mc{U}$ and  $\phi$ is $(\alpha,\beta,r)$-constrained, we have $\diam(\phi^{-1}(U)) \leq (D+\beta)/\alpha $. 
As $d_Y(U,U') > r$ for any distinct $U,U' \in \mc{U}$, we have $d_X(\phi^{-1}(U),\phi^{-1}(U'))>r$. 
Thus $\phi^{-1}(\mc{U})$ is an $(r, (D+\beta)/\alpha)$-witness for $\phi^{-1}(S)$.
\end{proof}	

\begin{lemma} \label{c:constrainedColor}	
Let $\alpha >0$, $\beta,r,D \geq 0$. 
Let $(X,d_X)$ and $(Y,d_Y)$ be finite metric spaces, and let $\phi: X \to Y$ be $(\alpha,\beta,r)$-constrained. 
If $c$ is an $(r,D)$-controlled $k$-coloring of $(Y,d_Y)$ with witness $(\mc{V}_i: i \in [k])$ for some positive integer $k$, then $c \circ \phi$ is an $(r, (D+\beta)/\alpha)$-controlled $k$-coloring of $(X,d_X)$ with witness $(\mc{U}_i: i \in [k])$, where $\mc{U}_i=\{\phi^{-1}(S): S \in \mc{V}_i, S \neq \emptyset\}$.
\end{lemma}

\begin{proof}
For every color $i \in [k]$, $\mc{V}_i$ is an $(r,D)$-witness for $c^{-1}(\{i\})$; by \cref{l:constrained}, the set $\phi^{-1}(c^{-1}(\{i\}))$ is $(r, (D+\beta)/\alpha)$-controlled with witness $(\mc{U}_i: i \in [k])$, where $\mc{U}_i=\{\phi^{-1}(S): S \in \mc{V}_i: S \neq \emptyset\}$.
It implies this lemma.
\end{proof}	

We convert results on coloring of graphs which admit certain kinds of tree-decomposition into the language of covers of metric spaces, starting with
the following weakening of \cite[Lemma 5.7]{Liu25} (taking $\ell=1$ and $\mu=r$).
 
\begin{lemma}[{{\cite[Lemma 5.7]{Liu25}}}] \label{center_bag_basic}
For any positive integers $k$ and $r$, there exists a positive integer $N = N_{\ref{center_bag_basic}}(k,r)$ such that if $G$ is a graph that has a $(k,r)$-centered tree-decomposition, then there exists a $2$-coloring of $G$ with weak diameter in $G$ at most $N$.
\end{lemma}

\begin{lemma} \label{center_bag_metric}
For all $R \geq r > 0$ and positive integer $k$, there exists a positive integer $D= D_{\ref{center_bag_metric}}(k,r,R)$ such that if $(X,d)$ is a finite metric space that has an $r$-compliant $(k,R)$-centered tree-decomposition $(T,\{B_t: t \in V(T)\})$, then there exists an $(r,D)$-controlled $2$-coloring of $(X,d)$ with an $(r,D)$-witness $(\mc{U}_1,\mc{U}_2)$ such that for any $t \in V(T)$ and $i \in \{1,2\}$, $B_t$ intersects at most $k$ elements of $\mc{U}_i$. 
\end{lemma}

\begin{proof}
Let $N=N_{\ref{center_bag_basic}}(k,1)$, and define $D=2NR$.

Let $J$ be a graph with vertex set $X$ such that for any distinct $x,y \in V(J)$, $xy \in E(J)$ if and only if $d(x,y) \leq 2R$ and $x,y \in B_t$ for some $t \in V(T)$. 
Since $(T,\{B_t: t \in V(T)\})$ is $(k,R)$-centered, we know that $(T,\{B_t: t \in V(T)\})$ is a tree-decomposition of $J$ such that every bag is a union of at most $k$ cliques.
Let $d'$ be the $(r,J)$-compression of $d$. 
Note that $(T,\{B_t: t \in V(T)\})$ is still an $r$-compliant tree-decomposition of $(X,d')$ and it is now $(k,r)$-centered. 
Let $G = G_r(X,d')$. 
By \cref{o:graphVSspaceTD}, $(T,\{B_t: t \in V(T)\})$ is a $(k,1)$-centered tree-decomposition of $G$. 
Moreover, every bag of $(T,\{B_t: t \in V(T)\})$ is a union of at most $k$ cliques in $G$.
By \cref{center_bag_basic}, there exists a $2$-coloring $c: V(G) \to \{1, 2\}$ of $G$ with weak diameter in $G$ at most $N=N_{\ref{center_bag_basic}}(k,1)$. 
By \cref{o:graphVSspaceColor}, there exists an $(r,Nr)$-controlled $2$-coloring of $(X,d')$. 
Consider the identity map $X \to X$ as a map between metric spaces $(X,d)$ and $(X,d')$. 
As this map is $(\frac{r}{2R},0,r)$-constrained, $c$ is an $(r,2NR)$-controlled 2-coloring of $(X,d)$ by \cref{c:constrainedColor}.
For every $i \in [2]$, let $\mc{U}_i$ be the set consisting of the vertex sets of the $c$-monochromatic components in $G$ of color $i$.
Then $(\mc{U}_1,\mc{U}_2)$ is an $(r,D)$-witness for $c$ (in terms of a coloring of $(X,d)$) by \cref{c:constrainedColor}.

Let $t \in V(T)$ and $i \in [2]$.
Note that every clique in $G$ intersects at most one $c$-monochromatic component with color $i$.
That is, every clique in $G$ intersects at most one element of $\U_i$.
Since $B_t$ is a union of at most $k$ cliques in $G$, we know that $B_t$ intersects at most $k$ elements of $\U_i$.
\end{proof}	

We will also need the following special case of \cite[Theorem 4.1]{CroLiu25} (taking the list-assignment of every vertex in $V(G)-Z$ to be $\{1,2\}$).

\begin{lemma}[{{\cite[Theorem 4.1]{CroLiu25}}}] \label{tw_basic}
For any positive integers $w$ and $s$, there exists a positive integer $N  = N_{\ref{tw_basic}}(w,s)$ such that for every graph $G$ of tree-width at most $w$ and every $Z \subseteq V(G)$, if $c_0$ is a coloring of $G[Z]$ with weak diameter in $G$ at most $s$, then $c_0$ can be extended to a coloring of $G$ with weak diameter in $G$ at most $N$ such that $c_0(v) \in \{1,2\}$ for every $v \in V(G)-Z$.
\end{lemma}

\begin{lemma} \label{centered_bag_color_0}
For any $R \geq r > 0$ and $k \in \mathbb{Z}_+$, there exists $D=D_{\ref{centered_bag_color_0}}(k,r,R)>0$ such that the following holds.
Let $(X,d)$ be a finite metric space that admits an $r$-compliant $(k,R)$-centered tree-decomposition.
Let $Z \subseteq X$, and let $c_Z: Z \rightarrow [2]$.
If each $c_Z^{-1}(\{i\})$ is $(r,0)$-controlled in $(X,d)$, then $c_Z$ extends to an $(r,D)$-controlled $2$-coloring of $(X,d)$. 
\end{lemma}

\begin{proof} 
Let $N  = N_{\ref{tw_basic}}(6k-1,1)$.
Let $D' = D_{\ref{center_bag_metric}}(k,r,R)$. 
Define $D=(Nr+D')(D'+r)/r$.

Let $(T,\{B_t: t \in V(T)\})$ be an $r$-compliant $(k,R)$-centered tree-decomposition of $(X,d)$. 
By \cref{center_bag_metric} there exists an $(r,D')$-controlled $2$-coloring $c'$ of $(X,d)$ with an $(r,D')$-witness $(\mc{U}_1,\mc{U}_2)$ such that every bag $B_t$ intersects at most $k$ elements of $\mc{U}_i$ for each $i \in [2]$. 
We choose such a witness with $|\mc{U}_1| + |\mc{U}_2|$ maximum.

Let $(X,d),Z,c_Z$ be as stated in the lemma.
Define $c_Z(x)=0$ for every $x \in X-Z$. 

We now construct a graph $G$ to which we will apply \cref{tw_basic}. 
Let $\mc{U}= \mc{U}_1 \cup \mc{U}_2$. 
Let $V(G) = \mc{U} \times \{0,1,2\}$ such that for any distinct vertices $(U,i),(U',j)$ of $G$, $(U,i)(U',j) \in E(G)$ if and only if there exist $x \in U $ and $x' \in U'$ such that $c_Z(x)=i,$ $c_Z(x')=j$ and $d(x,x') \leq r$.
For each $t \in V(T)$, let $B'_t = \{(U,i) \in V(G): U \in \mc{U}, U \cap B_t \neq \emptyset, i \in \{0,1,2\}\}$. 

\medskip

\noindent{\bf Claim 1:} $(T,\{B'_t: t \in V(T)\})$ is a tree-decomposition of $G$ of width at most $6k-1$. 

\noindent{\bf Proof of Claim 1:}
For every $t \in V(T)$, since every $B_t$ intersects at most $k$ elements of $\mc{U}_i$ for each $i \in [2]$, we know that $|B'_t| \leq 6k$.
For any $(U,i) \in V(G)$, since $U \cap B_t \neq \emptyset$ for some $t \in V(T)$, so $(U,i) \in B'_t$ for some $t \in V(T)$.
For any $(U,i)(U',j) \in E(G)$, there exist $x \in U$ and $x' \in U'$ such that $d(x,x') \leq r$, so there exists $t \in V(T)$ with $x,x' \in B_t$, and hence $B_t \cap U \neq \emptyset \neq B_t \cap U'$ and $(U,i),(U',j) \in B'_t$.
Therefore, it suffices to show that for every $v \in V(G)$, the set $\{t: v \in B'_t \}$ induces a subtree of $T$. 

Let $U \in \U$.
Let $B(U)=\{t: U \cap B_t \neq \emptyset\}$.
It suffices to show that the set $B(U)$ induces a subtree of $T$. 

Suppose to the contrary. 
Then there exist disjoint subsets $V_1$ and $V_2$ of $V(T)$ with $B(U) \subseteq V_1 \cup V_2$ such that no edge of $T$ joins $V_1$ to $V_2$, and $B(U)$ intersects both $V_1$ and $V_2$. 
For every $i \in [2]$, let $U_i = U \cap \bigcup_{t \in V_i}B_t$.
So $U_1$ and $U_2$ are non-empty, and $U=U_1 \cup U_2$.
As $(T,\{B_t: t \in V(T)\})$ is an $r$-compliant tree-decomposition of $(X,d)$, we have that $d(U_1,U_2) > r$. 
Thus we can replace $U$ by $\{U_1,U_2\}$ in $\mc{U}$ contradicting our initial choice of a witness.
This establishes the claim.
$\Box$

\medskip

Let $Z' = \mc{U} \times \{1,2\} \subseteq V(G)$.
For every $(U,i) \in Z'$, let $c_0((U,i))=i$. 
Since $c_Z^{-1}(\{i\})$ is $(r,0)$-controlled in $(X,d)$ for each $i$, we know that $c_0$ is a proper $2$-coloring of $G[Z']$. 
Thus \cref{tw_basic} implies that $c_0$ can be extended to a $2$-coloring $c_0$ of $G$ with weak diameter in $G$ at most $N$. 
Let $(Y,d_Y)$ be a metric space obtained from $G$ by rescaling metric by factor $r$; that is, $Y=V(G)$ and $d_Y(x,x')=r \cdot \dist_G(x,x')$ for all $x,y \in Y$. 
Then $G=G_r(Y,d_Y)$ and so $c_0$ is an $(r,Nr)$-controlled 2-coloring of $(Y,d_Y)$ by \cref{o:graphVSspaceColor}. 
	 
For each $x \in X$, let $U(x)$ be the member of $\mc{U}$ containing $x$.
Let $\phi(x)=(U(x),c_Z(x)) \in Y$ for every $x \in X$. 

\medskip

\noindent{\bf Claim 2:} $\phi : X \to Y$ is $(r/(D'+r),D',r)$-constrained. 

\noindent{\bf Proof of Claim 2:}
Let $x,x' \in X$.
Let $q = d_Y(\phi(x),\phi(x'))$.
So $\dist_G(\phi(x),\phi(x'))=q/r$.
Then there exists a path $Q_0Q_1...Q_{q/r}$ in $G$ such that $Q_0=\phi(x)$ and $Q_{q/r}=\phi(x')$.
For every $0 \leq \alpha \leq q/r$, denote $Q_\alpha$ by $(U_\alpha,i_\alpha)$.
For every $\alpha \in [q/r]$, since $Q_{\alpha-1}Q_\alpha \in E(G)$, there exists $b_{\alpha-1} \in U_{\alpha-1}$ and $a_\alpha \in U_\alpha$ such that $d(b_{\alpha-1},a_\alpha) \leq r$.
For every $0 \leq \alpha \leq q/r$, since $U_\alpha \in \U_1 \cup \U_2$, we have $\diam(U_\alpha) \leq D'$.
So $d(x,x') \leq D' \cdot (1+\frac{q}{r}) + r \cdot \frac{q}{r} = (D'+r)\frac{q}{r}+D'$.
Hence $d_Y(\phi(x),\phi(x')) = q \geq (d(x,x')-D') \cdot \frac{r}{D'+r} \geq \frac{r}{D'+r} \cdot d(x,x')-D'$.

If $d(x,x') \leq r$, then either $(U(x),c_Z(x)) = (U(x'),c_Z(x'))$ or $(U(x),c_Z(x))(U(x'),c_Z(x')) \in E(G)$, so $d_Y(\phi(x),\phi(x')) = r \cdot \dist_G(\phi(x),\phi(x')) \leq r$.
This shows that $\phi$ is $(r/(D'+r),D',r)$-constrained. 
$\Box$

\medskip

Thus $c_0 \circ \phi$ is an $(r,(Nr+D')(D'+r)/r)$-controlled 2-coloring of $(X,d)$ by \cref{c:constrainedColor}. 
Moreover, $c_0(\phi(z))= c_0(U(z),c_Z(z))=c_Z(z)$ for every $z \in Z$.
So $c_0 \circ \phi$ extends $c_Z|_Z$. 
This proves the lemma.
\end{proof}

Now we strengthen \cref{centered_bag_color_0} by allowing $c_Z$ to be $(r,D_0)$-controlled for any $D_0 \geq 0$.

\begin{lemma} \label{precolor_prepare}
For any real numbers $R \geq r > 0$ and $D_0 \geq 0$ and integer $k \in \bb{Z}_+$, there exists $D=D_{\ref{precolor_prepare}}(k,r,R,D_0)>0$ such that the following holds.
Let $(X,d)$ be a finite metric space that admits an $r$-compliant $(k,R)$-centered tree-decomposition.
Let $Z \subseteq X$ and $c_Z: Z \rightarrow [2]$.
If each $c_Z^{-1}(\{i\})$ is $(r,D_0)$-controlled in $(X,d)$, then $c_Z$ extends to an $(r,D)$-controlled $2$-coloring of $(X,d)$.
\end{lemma}

\begin{proof}
Let $D'= D_{\ref{centered_bag_color_0}}(k,r,R)$.
Define $D= (D'+ D_0)(D_0+r)/r$.

Let $(X,d),Z,c_Z$ be as stated in the lemma.
For $i \in [2]$, let $\U_i$ be a set witnessing that $c_Z^{-1}(\{i\})$ is $(r,D_0)$-controlled in $(X,d)$ such that $|\U_i|$ is maximum.
Let $\mc{U}= \mc{U}_1 \cup \mc{U}_2$. 

Let $X' = (X-Z) \cup \mc{U}$.
Let $\ell$ be a premetric on $X'$ such that $\ell(x,y) = \max\{r,d(x,y)\}$ for any distinct $x,y \in X'$.
Let $d'=d_{\ell}$ be the metric derived from $\ell$.  
For every $x \in X-Z$, let $\phi(x)=x$; for every $z \in Z$, there exists $U_z \in \U$ with $z \in U_z$, and we let $\phi(z)=U_z$.
So $\phi: X \to X'$ is a function.

\medskip

\noindent{\bf Claim 1:} $\phi$ is an $(r/(D_0+r),D_0,r)$-constrained map from $(X,d)$ to $(X',d')$. 

\noindent{\bf Proof of Claim 1:}
Let $x,x' \in X$.
Let $y_0,y_1,...,y_t$ be the sequence over $X'$ for some nonnegative integer $t$ such that $y_0=\phi(x)$, $y_t=\phi(x')$, and $d'(\phi(x),\phi(x'))=\sum_{i=0}^{t-1}\ell(y_i,y_{i+1})$.
Note that $d'(\phi(x),\phi(x'))=\sum_{i=0}^{t-1}\ell(y_i,y_{i+1}) \geq \max\{rt,\sum_{i=0}^{t-1}d(y_i,y_{i+1})\}$.
In particular, $t \leq \frac{d'(\phi(x),\phi(x'))}{r}$.
Since each $\U_i$ is $(r,D_0)$-controlled, 
	\begin{align*}
		d(x,x') & \leq (t+1)D_0 + \sum_{i=0}^{t-1}d(y_i,y_{i+1}) \\
		& \leq (\frac{d'(\phi(x),\phi(x'))}{r}+1)D_0+d'(\phi(x),\phi(x')) \\
		& = \frac{D_0+r}{r} d'(\phi(x),\phi(x'))+D_0.
	\end{align*}
So $d'(\phi(x),\phi(x')) \geq \frac{r}{D_0+r}(d(x,x')-D_0) \geq \frac{r}{D_0+r}d(x,x')-D_0$.

Note that $d'(\phi(x),\phi(x')) \leq \ell(\phi(x),\phi(x')) = \max\{r, d(\phi(x),\phi(x'))\} \leq \max\{r, d(x,x')\}$.
Hence, if $d(x,x') \leq r$, then $d'(\phi(x),\phi(x')) \leq r$.
$\Box$

\medskip

\noindent{\bf Claim 2:} If $x,x'$ are distinct elements of $X'$ with $d'(x,x') \leq r$, then $d(x,x') \leq r$.

\noindent{\bf Proof of Claim 2:}
Since $d'(x,x') \leq r$, there exist distinct $y_0,y_1,...,y_w \in X'$ for some positive integer $w$ such that $y_0=x$, $y_w=x'$ and $r \geq d'(x,x')=\sum_{i=0}^{w-1}\ell(y_i,y_{i+1}) = \sum_{i=0}^{w-1}\max\{r,d(y_i,y_{i+1})\} \geq wr$.
So $w=1$ and $\max\{r,d(y_0,y_{1})\}=r$.
Hence $r \geq d(y_0,y_1) = d(y_0,y_w)=d(x,x')$.
$\Box$

\medskip

Let $(T,\{B_t: t \in V(T)\})$ be an $r$-compliant $(k,R)$-centered tree-decomposition of $(X,d)$. 
For each $t \in V(T)$, let $B'_t=(B_t \cap (X-Z)) \cup \{U \in \mc{U}: U \cap B_t \neq \emptyset\}$. 

\medskip

\noindent{\bf Claim 3:} $(T,\{B'_t: t \in V(T)\})$ is an $r$-compliant $(k,R)$-centered tree-decomposition of $(X',d')$.

\noindent{\bf Proof of Claim 3:}
Since $(T,\{B_t: t \in V(T)\})$ is a tree-decomposition of $(X,d)$, it is clear that every point of $X'$ is in some bag of $(T,\{B'_t: t \in V(T)\})$.
For any distinct $x,x' \in X'$, since $d'(x,x') \leq \ell(x,x') = \max\{r,d(x,x')\} \leq \max\{R,d(x,x')\}$ and $B_t$ is $(k,R)$-centered in $(X,d)$ for every $t \in V(T)$, we know that $B'_t$ is $(k,R)$-centered in $(X',d')$ for every $t \in V(T)$.

Let $x,x' \in X'$ with $d'(x,x') \leq r$ and $x \neq x'$.
Then Claim 2 implies that $d(x,x') \leq r$.
Since $(T,\{B_t: t \in V(T)\})$ is an $r$-compliant tree-decomposition of $(X,d)$, there exists $t_0 \in V(T)$ such that $B_{t_0}$ intersects both $x$ and $x'$, and hence $B_{t_0}'$ contains both $x$ and $x'$.

Let $y \in X'$.
To prove this claim, it suffices to show $\{t \in V(T): y \in B'_t\}$ induces a subtree of $T$.
We may assume $y \in \U$ since $(T,\{B_t: t \in V(T)\})$ is a tree-decomposition of $(X,d)$.
So it suffices to show that the set $\{t \in V(T): y \cap B_t \neq \emptyset\}$, denoted by $B(y)$, induces a subtree of $T$.

Suppose to the contrary. 
Then there exist disjoint subsets $V_1$ and $V_2$ of $V(T)$ with $B(y) \subseteq V_1 \cup V_2$ such that no edge of $T$ joins $V_1$ to $V_2$, and $B(y)$ intersects both $V_1$ and $V_2$. 
For every $i \in [2]$, let $U_i = y \cap \bigcup_{t \in V_i}B_t$.
So $U_1$ and $U_2$ are non-empty, and $y=U_1 \cup U_2$.
As $(T,\{B_t: t \in V(T)\})$ is an $r$-compliant tree-decomposition of $(X,d)$, we have that $d(U_1,U_2) > r$. 
Thus we can replace $y$ by $\{U_1,U_2\}$ in $\mc{U}$, contradicting our initial choice of a witness.
This establishes the claim.
$\Box$

\medskip
	
For every $U \in \U$, let $c_0(U)$ be the unique $i \in [2]$ with $U \in \mc{U}_i$.
Then $c_0$ is a 2-coloring of $\U$ such that each $c_0^{-1}(\{i\})$ is $(r,0)$-controlled in $(X',d')$ by Claim 2. 
By Claim 3 and \cref{centered_bag_color_0}, $c_0$ can be extended to an $(r, D')$-controlled $2$-coloring $c'$ of $(X',d')$. 
As $\phi$ is $(r/(D_0+r),D_0,r)$-constrained by Claim 1, \cref{c:constrainedColor} implies that $c' \circ \phi$ is an $(r, (D'+ D_0)(D_0+r)/r)$-controlled 2-coloring of $(X,d)$.
Note that for every $z \in Z$, we know $c' \circ \phi(z) = c'(U_z)=c_0(U_z)=c_Z(z)$. 
Thus the lemma holds.
\end{proof}

\begin{lemma} \label{precolorSpace}
For any positive integers $k$ and real numbers $r,D_0,R > 0$ with $R \geq r$, there exists $D=D_{\ref{precolorSpace}}(k,r,D_0,R) > 0$ such that the following holds.
Let $(X,d)$ be a finite metric space. 
Let $Y \subseteq X$ be $(3r,D_0)$-controlled in $(X,d)$ such that the metric space $(X-Y,d|_{(X-Y) \times (X-Y)})$ admits an $r$-compliant $(k,R)$-centered tree-decomposition. 
Then there exists an $(r,D)$-controlled $2$-coloring $c: X \to [2]$ of $(X,d)$ such that $Y \subseteq c^{-1}(\{1\})$.  
\end{lemma}

\begin{proof} 
Let $D=\max\{D_0, D_{ \ref{precolor_prepare}}(k,r,R,D_0+2r) \}$.

Let $\mc{U}$ be a $(3r,D_0)$-witness for $Y$. 
For each $U \in \mc{U}$, let $N(U)= \{x \in X-Y: d(x,U) \leq r\}$. 
Let $\mc{N}= \{N(U): U \in \mc{U}\}$. 
Then $\diam(N) \leq D_0+2r$ for each $N \in \mc{N}$, and $d(N,N') > r$ for all distinct $N,N' \in \mc{N}$.  
	
Let $Z = \bigcup_{N \in \mc{N}}N$, and let $c_Z(z)=2$ for every $z \in Z$. 
Then $Z \subseteq X-Y$ and $c_Z: Z \rightarrow [2]$ is a function such that every $c_Z^{-1}(\{i\})$ is $(r,D_0+2r)$-controlled in $(X-Y,d|_{(X-Y) \times (X-Y)})$. 
Since $(X-Y,d|_{(X-Y) \times (X-Y)})$ admits an $r$-compliant $(k,R)$-centered tree-decomposition, \cref{precolor_prepare} implies that the coloring $c_Z$ extends to an $(r,D)$-controlled 2-coloring $c_0: X-Y \to [2]$ of $(X-Y,d|_{(X-Y) \times (X-Y)})$. 
We further extend $c_0$ to a 2-coloring $c$ of $(X,d)$ by setting $c(y)=1$ for every $y \in Y$. 

To prove this lemma, it suffices to show that $c$ is an $(r,D)$-controlled $2$-coloring of $(X,d)$.
For $i \in [2]$, let $\U_i$ be an $(r,D)$-witness for $c_0^{-1}(\{i\})$ in $(X-Y,d|_{(X-Y) \times (X-Y)})$.
Since $c^{-1}(\{2\}) \subseteq X-Y$, it is clear that $\U_2$ is an $(r,D)$-witness for $c^{-1}(\{2\})$.
Hence it suffices to show that $\mc{U} \cup \mc{U}_1$ is an $(r,D)$-witness for $c^{-1}(\{1\})$ in $(X,d)$.

It is clear that every member of $\U \cup \U_1$ has diameter in $(X,d)$ at most $D$.
Moreover, $d(U,U')>3r>r$ for any $U,U' \in \U$ since $\U$ is a $(3r,D_0)$-witness for $Y$; $d(U,U')>r$ for any $U,U' \in \U_1$ since $\U_1$ is an $(r,D)$-witness for $c_0^{-1}(\{1\})$.
If $U \in \U$ and $U' \in \U_1$, then since $U' \subseteq c^{-1}(\{1\})$ and $c(z)=c_Z(z)=2$ for every $z \in Z \supseteq N(U)$, we know that $U' \cap N(U)=\emptyset$ by the definition of $c_Z$, so $d(U',U)>r$.
This shows that $\mc{U} \cup \mc{U}_1$ is an $(r,D)$-witness for $c^{-1}(\{1\})$.
\end{proof}

Now we prove the main result of this section.

\begin{lemma} \label{colorSpace}
For any positive integers $k, s$ and any real numbers $R \geq r > 0$, there exists $D=D_{\ref{colorSpace}}(k,r,R,s) > 0$ such that if $(X,d)$ is a finite metric space that has a partition $(X_1,\ldots, X_s)$ such that $(X_i,d|_{X_i \times X_i})$ admits an $3^{s-1}r$-compliant $(k,R)$-centered tree-decomposition for every $i \in [s]$, then there exists an $(r,D)$-controlled $(s+1)$-coloring of $(X,d)$.
\end{lemma}

\begin{proof}
Let $k,s$ be positive integers and let $r,R$ be positive real numbers.
For any positive real numbers $x,y$ with $y \geq x$, 
	\begin{itemize}
		\item let $D_1(k,x,y) = D_{\ref{center_bag_metric}}(k,x,y)$, and
		\item for every $i \geq 1$, let $D_{i+1}(k,x,y)=D_{\ref{precolorSpace}}(k,r,D_i(k,3x,3y),R)$.
	\end{itemize}
Define $D=D_s(k,r,R)$.

We prove the theorem by induction on $s$. 
The base case $s=1$ is a consequence of \cref{center_bag_metric}.

So we may assume $s \geq 2$.
Note that every $(k,R)$-controlled tree-decomposition is also $(k,3R)$-controlled.
So $(X_1,\ldots,X_{s-1})$ is a partition of $(X-X_s,d|_{(X-X_s) \times (X-X_s)})$ such that for every $i \in [s-1]$, $(X_i,d|_{X_i \times X_i})$ admits a $3^{s-1}r$-compliant (and hence $(3^{s-2} \cdot 3r)$-compliant) $(k,3R)$-controlled tree-decomposition.
We apply the inductive hypothesis to $(X-X_s,d|_{(X-X_s) \times (X-X_s)})$ with $3r$ instead of $r$. 
Then there exists a $(3r, D_{s-1}(k,3r,3R))$-controlled $s$-coloring $c_1:X -X_s \to [s]$ of $(X-X_s,d|_{(X-X_s) \times (X-X_s)})$. 

Let $D_0=D_{s-1}(k,3r,3R)$.
Let $Y= c_1^{-1}(\{1\})$. 
Then $Y$ is $(3r, D_0)$-controlled in both $(X-X_s,d|_{(X-X_s) \times (X-X_s)})$ and $(X,d)$. 
So $Y$ is $(3r,D_0)$-controlled in $(X_s \cup Y, d|_{(X_s \cup Y) \times (X_s \cup Y)})$.
Since $(X_s,d|_{X_s \times X_s})$ admits a $3^{s-1}r$-compliant (and hence $r$-compliant) $(k,R)$-centered tree-decomposition, \cref{precolorSpace} implies that there exists an $(r,D)$-controlled $2$-coloring $c_s: X_s \cup Y \to \{1,s+1\}$ of $(X_s \cup Y, d|_{(X_s \cup Y) \times (X_s \cup Y)})$ such that $Y \subseteq c_s^{-1}(\{1\})$.  
As $c_1$ agrees with $c_s$ on $Y$, putting $c_1$ and $c_s$ together gives an $(r,D)$-controlled $(s+1)$-coloring of $(X,d)$. 
\end{proof}

\section{Asymptotic dimension of space-filling families} \label{sec:asdim_space_filling}

For a family $\Se$ and a real number $b>0$, we say that a set $S$ is a \emph{$b$-shallow union} of elements of $\mc{S}$ if there exists $\mc{S}'\subseteq \mc{S}$ such that $S=\bigcup_{X \in \Se'}X$ and $I(\mc{S}')$ has radius at most $b$.

\begin{lemma}[{{\cite[Lemma 6]{DvoNor25}}}] \label{lem-cons}
For a function $f:\mathbb{R}_+\to\mathbb{Z}_+$, let $\mc{S}$ be an $f$-space-filling family of subsets of a metric space $(X,d)$.
Let $b >0$, and let $h(x)=f((2b+2)(x+1))$ for every $x>0$.
If $\mc{S}'$ is a family of $b$-shallow unions of elements of $\mc{S}$, then $\mc{S}'$ is $h$-space-filling.
\end{lemma} 

The following is a simple combination of ~\cite[Lemma 7]{DvoNor25} and ~\cite[Lemma 8]{DvoNor25}, and we include a proof for completeness, though we omit the formal definition of several terminologies in \cite{DvoNor25}.

\begin{lemma} \label{l:dn1}
For any function $f:\mathbb{R}_+\to\mathbb{Z}_+$ and dilation $g:\mathbb{R}_+\to\mathbb{R}_+$, there exists a positive integer $k = k_{\ref{l:dn1}}(f,g)$ such that the following holds. 
Let $n$ be a nonnegative integer, and let $(X,d)$ be a metric space of Assouad-Nagata dimension at most $n$ such that $g$ is an $n$-dimensional control function.
Let $\mc{S}$ be a finite $f$-space-filling family of subsets of $X$, and let $G$ be the intersection graph of $\mc{S}$. 
Then there exists a partition $(\mc{S}_1,\ldots,\mc{S}_{n+1})$ of $\mc{S}$ such that for every $i \in [n+1]$, $G[\mc{S}_i]$ admits a $(k,2)$-centered tree-decomposition. 
\end{lemma}

\begin{proof}
We may assume that every member of $\Se$ is a nonempty set since empty members of $\Se$ correspond to isolated vertices of $G$.
By \cite[Lemma 7]{DvoNor25}, there exists a real number $C>0$ only depending on $g$ such that $(X,d)$ admits an $(n+1)$-laminar $C$-web.
That is, there exist $n+1$ laminar families $\F_1,\F_2,...,\F_{n+1}$ such that for every $S \in \Se$, there exists $i_S \in [n+1]$ such that $\F_{i_S}$ $C$-catches $S$.
For every $i \in [n+1]$, let $\Se_i = \{S \in \Se: i_S=i\}$.
So $(\mc{S}_1,\ldots,\mc{S}_{n+1})$ is a partition of $\mc{S}$.

Let $i \in [n+1]$.
It suffices to show that $G[\mc{S}_i]$ admits a $(k,1)$-centered tree-decomposition. 
Since $\Se_i \subseteq \Se$ and $\Se$ is $f$-space-filling, we know that $\Se_i$ is $f$-space-filling.
By \cite[Lemma 8]{DvoNor25}, $I(\Se_i)$ has an $((f(C),1))$-centered tree-decomposition.
Note that $G[\Se_i]=I(\Se_i)$.
Then this lemma follows by taking $k=f(C)$.
Note that $k$ only depends on $f$ and $g$.
\end{proof}

\begin{lemma} \label{finite_space_filling_coloring}
For any positive integer $n$, function $f: {\mathbb R}_+ \rightarrow {\mathbb Z}_+$, dilation $g:\mathbb{R}_+\to\mathbb{R}_+$, and real number $r>0$, there exists $D = D_{\ref{finite_space_filling_coloring}}(n,f,g,r)>0$ such that the following holds.
Let $(X,d)$ be a metric space of Assouad-Nagata dimension at most $n$ such that $g$ is an $n$-dimensional control function.
If $\Se$ is a finite $f$-space filling family of subsets of $X$, then the metric space $(\Se,\dist_{I(\Se)})$ has an $(r,D)$-controlled $(n+2)$-coloring.
\end{lemma}

\begin{proof} 
Let $n,f,g,r$ be as stated in the lemma.
Let $b=3^{n}r$. 
Let $h:\mathbb{R}_+\to\mathbb{Z}_+$ be the function stated in \cref{lem-cons}.
Note that $h$ only depends on $f$ and $b$ and hence only depends on $f,n,r$.
Let $k=k_{\ref{l:dn1}}(h,g)$.
Define $D=D_{\ref{precolorSpace}}(k, r,4b+2,n+1)$.

Let $(X,d)$ and $\Se$ be as stated in the lemma.
For $S \in \mc{S}$, let $U(S)$ be the union of all $S' \in S$ such that $\dist_{I(\Se)}(S,S') \leq b$. 
Then $U(S)$ is a $b$-shallow union of some elements of $\mc{S}$. 
Let $\mc{S}' = \{U(S): S \in \mc{S}\}$. 
By \cref{lem-cons}, the system $\mc{S}'$ is $h$-space-filling. 
Let $G'$ be the intersection graph of $\mc{S}'$. 
By \cref{l:dn1}, there exists a partition $(\mc{S}'_1,\ldots,\mc{S}'_{n+1})$ of $\mc{S}'$ such that for every $i \in [n+1]$, $G'[\mc{S}'_i]$ admits a $(k,2)$-centered tree-decomposition $(T_i, \{B'_t: t \in V(T_i)\})$. 

For every $i \in [n+1]$, let $\mc{S}_i = \{S \in \mc{S}: U(S) \in \mc{S}'_i\}$. 
Then $(\Se_1,\Se_2,...,\Se_{n+1})$ is a partition of $\Se$.
For any $i \in [n+1]$ and $t \in V(T_i)$, let $B_t = \{S \in \mc{S}_i: U(S) \in B'_t\}$.

\medskip

\noindent{\bf Claim 1:} For any $i \in [n+1]$, $(T_i, \{B_t: t \in V(T_i)\})$ is a $b$-compliant $(k,4b+2)$-centered tree-decomposition of the metric space $(\mc{S}_i,\dist_{I(\Se)}|_{\Se_i})$.

\noindent{\bf Proof of Claim 1:}
For every $S \in \Se_i$, the set $\{t \in V(T_i): S \in B_t\} = \{t \in V(T_i): U(S) \in B'_t\}$ induces a non-null subtree of $T_i$.
For any $S,S' \in \Se_i$, if $\dist_{I(\Se)}(S,S') \leq b$, then either $U(S)=U(S')$, or $U(S)$ and $U(S')$ are adjacent in $G'[\mc{S}'_i]$, so there exists $t \in V(T_i)$ with $U(S),U(S') \in B'_t$, and hence $S,S' \in B_t$. 
So $(T_i, \{B_t: t \in V(T_i)\})$ is a $b$-compliant tree-decomposition of the metric space $(\mc{S}_i,\dist_{I(\Se)}|_{\Se_i})$.

For any $S,S' \in \Se_i$, if $U(S)$ and $U(S')$ are adjacent in $G'[\mc{S}'_i]$, then $\dist_{I(\Se)}(S,S') \leq 2b+1$.
For every $t \in V(T_i)$, since $B'_t$ is $(k,2)$-centered in $G'$, we know that $B_t$ is $(k,4b+2)$-centered in $(\mc{S}_i,\dist_{I(\Se)}|_{\Se_i})$.
$\Box$

\medskip

By Claim 1 and \cref{colorSpace}, $(\mc{S},\dist_{I(\Se)})$ admits an $(r,D)$-controlled $(n+2)$-coloring, as desired.
\end{proof}	

Now we are ready to prove \cref{asdim_space_filling_intro}.
We restate it here for reader's convenience.

\begin{theorem} \label{asdim_space_filling}
For any positive integer $n$, function $f: {\mathbb R}_+ \rightarrow {\mathbb Z}_+$ and metric space $(X,d)$ of Assouad-Nagata dimension at most $n$, if $\F$ is the class of the intersection graphs of $f$-space filling families of subsets of $X$, then the asymptotic dimension of $\F$ is at most $n+1$.
\end{theorem}

\begin{proof}
Let $n,f,(X,d)$ be as stated in the theorem.
Let $g$ be a dilation that is an $n$-dimensional control function of $(X,d)$.
For every $r \in \mathbb{Z}_+$, let $h(r)= \lceil D_{\ref{finite_space_filling_coloring}}(n,f,g,r) \rceil$.
By \cref{l:asweak}, there exists a function $\eta: \mathbb{Z}_+ \rightarrow \mathbb{Z}_+$ only depending on $n$ and $h$ such that $\eta$ is an $(n+1)$-dimensional control function for a graph class ${\mathcal G}$ if for any $r \in \mathbb{Z}_+$ and $G \in {\mathcal G}$, the graph $G^r$ has an $(n+2)$-coloring of weak diameter in $G^r$ at most $h(r)$.
Let $h^*$ be the function $g$ stated in \cref{finite_asdim} by replacing $f$ in \cref{finite_asdim} by $\eta$.

For every finite $f$-space filling family $\Se$ of subsets of $X$ and for every positive integer $r$, \cref{finite_space_filling_coloring} implies that there exists an $(r,h(r))$-controlled $(n+2)$-coloring of the metric space $(\Se,\dist_{I(\Se)})$, so there exists a $(n+2)$-coloring of $(I(\Se))^r$ of weak diameter in $I(\Se)$ (and hence in $(I(\Se))^r$) at most $h(r)$.
Hence $\eta$ is an $(n+1)$-dimensional control function of the set of all intersection graphs of finite $f$-space filling families of subsets of $X$. 
For every (finite or infinite) graph $G$ that is the intersection graph of an $f$-space filling family of subsets of $X$, every finite induced subgraph of $G$ is an intersection graph of a finite $f$-space filling family of subsets of $X$, so \cref{finite_asdim} implies that $h^*$ is an $(n+1)$-control function of $G$.
Since $h^*$ is independent from $G$, we know that $h^*$ is an $(n+1)$-dimensional control function of the set of all intersection graphs of $f$-space filling families of subsets of $X$.
This proves the theorem.
\end{proof}

\section{Tightness of bounds} \label{sec:lower_bound} 

In this section we prove \cref{asdim_space_filling_lower_intro}, which shows that the bound on asymptotic dimension of space-filling families established in \cref{asdim_space_filling} is optimal.

Our proof uses a discrete variant of the following classical Knaster--Kuratowski--Mazurkiewicz (KKM) theorem~\cite{KKM29}.

\begin{theorem}[Knaster--Kuratowski--Mazurkiewicz (KKM) theorem~\cite{KKM29}]\label{t:KKM}
Let $X_1,\ldots, X_m$ be a cover of a geometric $(n+1)$-dimensional simplex $\Delta^{n+1}$ by closed sets such that no $n+2$ sets have a common point. Then $X_i$ intersects every facet of $\Delta^{n+1}$ for some  $1 \leq i \leq m$.
\end{theorem}

We now present the construction of  the space-filling families that we study in this section. 

Let $n$ be a fixed positive integer for the remainder of the section. 
Moreover, we use $| \cdot |$ to denote the 2-norm and use $B(x,r)$ to denote the ball in $\mathbb{R}^n$ of radius $r$ (i.e.\ $\{y \in \bb{R}^n: |x-y| \leq r\}$).

Let $$\Delta = \{x\in \bb{R}_{\geq 0}^{n} : \| x \|_1   \leq 1\}$$ be the geometric $n$-dimensional simplex.
Let $P_{-1} = \emptyset$ and $Q_{-1} = \emptyset$. 
For an integer $i \geq 0$, let  $$Q_i = \{x \in \Delta :  2^{i} x \in \bb{Z}_{+}^{n}\} \qquad \mathrm{and} \qquad P_i = Q_i - Q_{i-1}.$$ 
For $i \geq 0$ and  $p \in P_i$, let $$S(p) = \{ p' \in P_i \cup P_{i-1} : |p-p'| \leq  \sqrt{n}2^{1-i} \}.$$
For every $i \geq 0$, let $$ \mc{S}_i =  \{ S(p): p \in P_i\}.$$
Finally, let $$\mc{S}= \bigcup_{i \in \bb{Z}_{+}} \mc{S}_i.$$

We will show that \begin{itemize} \item there exists a function $f=f_n: \bb{R}_+ \to \bb{R}_+$ such that $\mc{S}$ is $f$-space-filling (\cref{l:space}), and that
	\item The $(n+1)$-colorings of the intersection graphs of subsets of $\mc{S}$ have monochromatic components of arbitrarily large weak diameter (\cref{l:space2}). 
\end{itemize}
\cref{asdim_space_filling_lower_intro} follows from these.

First, we need a simple estimate on the number of points of our sets $P_i$ and $Q_i$ in balls of given radius.

\begin{lemma}\label{l:balls}  
Let $i \geq 0$ be an integer and $r > 0 $.
Then for every $x \in \Delta$, the following statements hold.
    \begin{enumerate} 
        \item If $r \geq \sqrt{n}2^{-i}$, then $ B(x,r) \cap P_i \neq \emptyset$.
        \item $|B(x,r) \cap Q_i| \leq (2^{i+1}r+1)^n$.
	\end{enumerate} 	
\end{lemma}

\begin{proof} 
We first prove Statement 1.
The claim is easy to verify for $n=1$ and so we assume $n>1$.
If $x \in Q_i$, then there exists $1 \leq j \leq n$ such that $x + 2^{-i}e_j$ or $x - 2^{-i}e_j$ is in $Q_i - Q_{i-1}=P_i$ and we are done.
So we may assume $x \not \in Q_i$.

Hence there exists a cube $C$ in $\mathbb{R}^n$ with vertices in $Q_i$ and side length $2^{-i}$ such that $C$ contains $x$.
Since $x \not \in Q_i$, $x$ is not a vertex of $C$.
So $B(x, \sqrt{n}2^{-i})$ contains two distinct vertices of $C$, which are points of $Q_i$.
Let $q,q'$ be distinct points in $B(x, \sqrt{n}2^{-i}) \cap Q_i$ such that $| q - q' |$ is minimum. If $q$ or $q'$ lie in $  P_i = Q_i \setminus Q_{i-1}$ the claim holds. Otherwise, $q,q' \in Q_{i-1}$, which implies $(q+q')/2 \in Q_i \cap B(x, \sqrt{n}2^{-i})$ contradicting the choice of $q$ and $q'$ and finishing the proof.
	
Now we prove Statement 2.
	Note that the balls of radius $2^{-i-1}$ centered at points of $Q_i$ are internally disjoint. Let $m= |B(x,r) \cap Q_i|$. Then $B(x, r+2^{-i-1})$ contains $m$ internally disjoint balls of radius  $2^{-i-1}$. As the volume of an $n$-dimensional ball is proportional to $n$th power of its radius we have $(r+2^{-i-1})^n \geq m (2^{-i-1})^n$, as desired.
\end{proof}

\begin{lemma}\label{l:space} The family $\mc{S}$ is $f$-space-filling, where $f(x)=(8\sqrt{n}x+1)^n$.
\end{lemma}

\begin{proof} 
For $p \in \Delta$ and $s,r \in \bb{R}^{+} $ we need to bound the number of pairwise disjoint elements of $\mc{S}$ intersecting $B(p,r)$ with diameter at least $s$. 
Let $i_0$ be the maximum integer such that $ \sqrt{n}2^{2-i_0} \geq s$.
If the diameter of an element $S \in \mc{S}$ is at least $s$, then $S \in  \mc{S}_i $ for some $i$ with $ \sqrt{n}2^{2-i} \geq s$, so $S$ contains a point in $Q_{i_0}$. 
By Statement 2 of \cref{l:balls}, there are at most  $$|B(p,r) \cap Q_{i_0}| \leq (2^{i_0+1}r+1)^n  \leq (\fs{8\sqrt{n}r}{s}+1)^n = f(\frac{r}{s})$$ pairwise disjoint elements of $\mc{S}$ of diameter at least $s$ intersecting $B(p,r)$.
\end{proof}	

\begin{lemma}\label{l:space2} 
For every $\ell>0$, there exists $k \in \mathbb{Z}$ such that every $(n+1)$-coloring of $I(\bigcup_{i=0}^k \mc{S}_i)$ has a monochromatic component of weak diameter at least $\ell$.
\end{lemma}

\begin{proof} 
The cone $\Delta'$ over $\Delta$ is obtained from $\Delta \times [0,1]$ by identifying all points in $\Delta \times \{0\}$ to a single point. 	We describe a cover  of $\Delta'$ to which we apply \cref{t:KKM}, working first with $\Delta \times [0,1]$.
 	
Let $k$ be a positive integer such that $\min \{k/2,  2^{k/2}n^{-3/2}\}>\ell$.
	
For an integer $0 \leq i \leq k$, let $I_i =  \left[ \frac{i}{k+1}, \frac{i+1}{k+1} \right].$
For a point $p \in P_i$, let $$Z(p)= B(p,\sqrt{n}2^{-i}) \times I_i.$$ 
By Statement 1 of \cref{l:balls}, for every $0 \leq i \leq k$, we have \begin{equation}\label{e:cover} \Delta \times I_i = \bigcup_{p \in P_i} Z(p) \end{equation}
Moreover, by definition of $\mc{S}(p)$, for any $p,p' \in Q_k$, if $Z(p) \cap Z(p') \neq \emptyset$ then $S(p)  \cap S(p') \neq \emptyset$.

Consider now an arbitrary $(n+1)$-coloring of $I(\bigcup_{i=0}^k \mc{S}_i)$. For a monochromatic component $C$ of this coloring define $X_C = \bigcup_{p \in V(C)} Z(p)$. By \eqref{e:cover} the sets $X_C$ cover $\Delta \times [0,1]$, and thus $\Delta'$ after the identification. By the second observation in the previous paragraph the sets corresponding to distinct components of the same color are disjoint, and this property is preserved by the identification, as every two sets in $\mc{S}_0$ are adjacent in the intersection graph.

Thus the cover $\{X_C\}$ of $\Delta'$ satisfies the conditions of \cref{t:KKM}.

By \cref{t:KKM}  there exists a monochromatic component $C$ such that $X_C$ intersects every facet of $\Delta'$. As the facet $\Delta \times \{1\}$ of $\Delta'$ is only covered by sets $Z(p)$ for $p \in P_k$, there exists $s \in V(C) \cap \mc{S}_k$.
If $C$ contains a point  $s' \in \mc{S}_i$ for $i \leq k/2$ then $C$ has weak diameter at least $k/2$, as $\dist_{I(\bigcup_{i=0}^k \mc{S}_i)}(s,s') \geq k-i \geq k/2$. 

We now assume that  $V(C) \subseteq \bigcup_{i > k/2} \mc{S}_i$. Let $Y_C$ be the projection of $X_C$ on $\Delta$. The diameter of the projection of every set $Z(p)$ comprising $X_C$ is at most $\sqrt{n}2^{-k/2}$. However, $Y_C$ intersects every facet of $
\Delta$ and thus has diameter at least $1/n$. It follows that the weak diameter of $C$ is at least $\frac{1/n}{\sqrt{n}2^{-k/2}}= 2^{k/2}n^{-3/2}$. 
We have shown that every  $(n+1)$-coloring of $I(\bigcup_{i=0}^k \mc{S}_i)$ has a monochromatic component of weak diameter at least $$\min \{k/2,  2^{k/2}n^{-3/2}\}>\ell,$$
as desired.
\end{proof}	

For every $k$, $\bigcup_{i=0}^k \mc{S}_i$ is a subset of $\Se$, so it is $f$-space-filling by \cref{l:space}.
Hence \cref{asdim_space_filling_lower_intro} follows from \cref{l:space2,l:asweak}.

\bigskip

\noindent{\bf Acknowledgement:} This work was initiated at ``New Perspectives in Colouring and Structure (24w5272)'' held at the Banff International Research Station in 2024.
We thank the organizers for creating a nice work environment and thank Agelos Georgakopoulos for discussion about \cref{sphere_intro}.

\end{document}